\documentclass[a4paper,10pt]{amsart}
\title[Topology of Springer fibers of type C and D]{Topology of two-row Springer fibers for the even orthogonal and symplectic group}
\author{Arik Wilbert}
\date{\today}
\address{Mathematical Institute\\
	University of Bonn\\
	Endenicher Allee 60\\
	53115 Bonn\\
	Germany}
\email{wilbert@math.uni-bonn.de}
\thanks{This research was funded by a Hausdorff scholarship of the Bonn International Graduate School in Mathematics}

\usepackage{geometry}
\usepackage[english]{babel}
\usepackage{amsfonts}
\usepackage{amssymb}
\usepackage{enumerate}
\usepackage[noadjust]{cite}
\usepackage{tikz}
\usetikzlibrary{arrows,decorations.pathmorphing}
\usepackage[all, arrow, matrix, curve, knot, poly]{xy}
\usepackage{tikz-cd}
\usepackage{youngtab}
\usepackage{hyperref}

\theoremstyle{definition}
\newtheorem{defi}{Definition}
\theoremstyle{remark}
\newtheorem{ex}[defi]{Example}
\newtheorem{rem}[defi]{Remark}
\newtheorem{notation}[defi]{Notation}
\theoremstyle{plain}
\newtheorem{lem}[defi]{Lemma}

\newtheorem{prop}[defi]{Proposition}
\newtheorem{thm}[defi]{Theorem}
\newtheorem{introthm}{Theorem}
\newtheorem*{introprop}{Proposition}

\newcommand{\CupConnect}{\text{\----}}
\newcommand{\DCupConnect}{\text{\----}\hspace{-10pt}\bullet\hspace{2pt}}
\newcommand{\RayConnect}{\text{\----}\hspace{-6pt}\shortmid}
\newcommand{\DRayConnect}{\text{\----}\hspace{-8pt}\bullet\hspace{-3pt}\shortmid}
\newcommand{\ba}{{\bf a}}
\newcommand{\bb}{{\bf b}}

\DeclareMathOperator{\spn}{span}
\DeclareMathOperator{\im}{im}

\begin{document}

\begin{abstract}
We construct an explicit topological model (similar to the topological Springer fibers appearing in work of Khovanov and Russell) for every two-row Springer fiber associated with the even orthogonal group and prove that the respective topological model is homeomorphic to its corresponding Springer fiber. This confirms a conjecture by Ehrig and Stroppel concerning the topology of the equal-row Springer fiber for the even orthogonal group. Moreover, we show that every two-row Springer fiber for the symplectic group is homeomorphic (even isomorphic as an algebraic variety) to a connected component of a certain two-row Springer fiber for the even orthogonal group.    
\end{abstract}

\maketitle

\section{Introduction}

In \cite{Kho04} Khovanov introduced a topological model for all Springer fibers of type $A$ corresponding to nilpotent endomorphisms with two equally sized Jordan blocks as a means of showing that the cohomology rings of these Springer fibers are isomorphic to the center of the algebras appearing in his groundbreaking work on the categorification of the Jones polynomial~\cite{Kho00,Kho02}. He also conjectured that the topological models are in fact homeomorphic to the corresponding Springer fibers~\cite[Conjecture 1]{Kho04}. This conjecture was proven independently by Wehrli~\cite{Weh09} and Russell-Tymoczko~\cite[Appendix]{RT11} using results contained in \cite{CK08}. The constructions and results were generalized to all two-row Springer fibers of type $A$ in \cite{Rus11}. In this article we define topological models for all two-row Springer fibers associated with the even orthogonal (type $D$) and the symplectic group (type $C$) and prove that they are homeomorphic to their corresponding Springer fiber.  

We fix an even positive integer $n=2m$ and let $\beta_D$ (resp.\ $\beta_C$) be a nondegenerate symmetric (resp.\ symplectic) bilinear form on $\mathbb C^n$ and let $O(\mathbb C^n,\beta_D)$ (resp.\ $Sp(\mathbb C^n,\beta_C)$) be the corresponding isometry group with Lie algebra $\mathfrak{so}(\mathbb C^n,\beta_D)$ (resp.\ $\mathfrak{sp}(\mathbb C^n,\beta_C)$). The group $O(\mathbb C^n,\beta)$ (resp.\ $Sp(\mathbb C^n,\beta_C)$) acts on the affine variety of nilpotent elements $\mathcal N_D\subseteq\mathfrak{so}(\mathbb C^n,\beta_D)$ (resp.\ $\mathcal N_C\subseteq\mathfrak{sp}(\mathbb C^n,\beta_C)$) by conjugation and it is well known that the orbits under this action are in bijective correspondence with partitions of $n$ in which even (resp.\ odd) parts occur with even multiplicity \cite{Wil37,Ger61}. The parts of the partition associated to the orbit of an endomorphism encode the sizes of the Jordan blocks in Jordan normal form. 

Given a nilpotent endomorphism $x\in\mathcal N_D$, the associated {\it (algebraic) Springer fiber} $\mathcal Fl^{x}_D$ {\it of type $D$} is defined as the projective variety consisting of all full isotropic (with respect to $\beta_D$) flags $\{0\}=F_0\subsetneq F_1 \subsetneq \ldots \subsetneq F_m$ in $\mathbb C^n$ which satisfy the condition $xF_i \subseteq F_{i-1}$ for all $i \in \{1,\ldots,m\}$. Analogously one defines the Springer fiber of type $C$ (simply replace all the $D$'s in the definition by $C$'s). These varieties naturally arise as the fibers of a resolution of singularities of the nilpotent cone, see e.g.\ \cite[Chapter 3]{CG97}. In general they are not smooth and decompose into many irreducible components. 

The goal is to understand the topology of the irreducible components of the Springer fibers and their intersections explicitly and provide a combinatorial description. In general this is a very difficult problem (even in type $A$). Thus, we restrict ourselves to {\it two-row Springer fibers}, i.e.\ we only consider endomorphisms of Jordan type $(n-k,k)$, where $k\in\{1,\ldots,m\}$. Note that for type $D$ (resp. type $C$) that means that either $k=m$ or $k$ is odd (resp. even). Since the Springer fiber depends (up to isomorphism) only on the conjugacy class of the chosen endomorphism it makes sense to speak about the $(n-k,k)$ Springer fiber, denoted by $\mathcal Fl^{n-k,k}_D$ (resp. $\mathcal Fl^{n-k,k}_C$), without further specifying the nilpotent endomorphism. 

Henceforth, we fix a two-row partition $(n-k,k)$ labelling a nilpotent orbit of type $D$. Note that every two-row partition for an orbit of type $C$ arises from a two-row partition of type $D$ by subtracting $1$ in both parts of the type $D$ partition. 

Consider a rectangle in the plane together with a finite collection of vertices evenly spread along the upper horizontal edge of the rectangle. A {\it cup diagram} is a non-intersecting diagram inside the rectangle obtained by attaching lower semicircles called {\it cups} and vertical line segments called {\it rays} to the vertices. We require that every vertex is joined with precisely one endpoint of a cup or ray. Moreover, a ray always connects a vertex with a point on the lower horizontal edge of the rectangle. Additionally, any cup or ray for which there exists a path inside the rectangle connecting this cup or ray to the right edge of the rectangle without intersecting any other part of the diagram may be equipped with one single dot. We do not distinguish between diagrams which are related by a planar isotopy fixing the boundary. Here are two examples: 
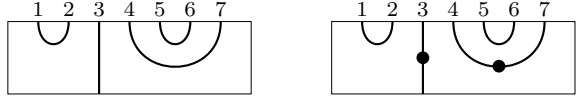
\[
\begin{tikzpicture}[scale=.8]
\draw[very thin] (-.5,0) -- (3.5,0) -- (3.5,-1.2) -- (-.5,-1.2) -- cycle;
\draw[thick] (0,0) .. controls +(0,-.5) and +(0,-.5) .. +(.5,0);
\draw[thick] (1.5,0) .. controls +(0,-1) and +(0,-1) .. +(1.5,0);
\draw[thick] (2,0) .. controls +(0,-.5) and +(0,-.5) .. +(.5,0);

\draw[thick] (1,0) -- +(0,-1.2);
\begin{footnotesize}
\node at (0,.2) {1};
\node at (.5,.2) {2};
\node at (1,.2) {3};
\node at (1.5,.2) {4};
\node at (2,.2) {5};
\node at (2.5,.2) {6};
\node at (3,.2) {7};
\end{footnotesize}
\end{tikzpicture}
\hspace{3em}
\begin{tikzpicture}[scale=.8]
\draw[very thin] (-.5,0) -- (3.5,0) -- (3.5,-1.2) -- (-.5,-1.2) -- cycle;
\draw[thick] (0,0) .. controls +(0,-.5) and +(0,-.5) .. +(.5,0);
\draw[thick] (1.5,0) .. controls +(0,-1) and +(0,-1) .. +(1.5,0);
\draw[thick] (2,0) .. controls +(0,-.5) and +(0,-.5) .. +(.5,0);

\fill[thick] (2.25,-.74) circle(3pt);

\draw[thick] (1,0) -- +(0,-1.2);
\fill[thick] (1,-.6) circle(3pt);
\begin{footnotesize}
\node at (0,.2) {1};
\node at (.5,.2) {2};
\node at (1,.2) {3};
\node at (1.5,.2) {4};
\node at (2,.2) {5};
\node at (2.5,.2) {6};
\node at (3,.2) {7};
\end{footnotesize}
\end{tikzpicture}
\]

Let $\mathbb B^{n-k,k}$ denote the set of all cup diagrams on $m$ vertices with $\lfloor\frac{k}{2}\rfloor$ cups. This set decomposes as a disjoint union $\mathbb B^{n-k,k} = \mathbb B^{n-k,k}_{\mathrm{even}} \sqcup \mathbb B^{n-k,k}_{\mathrm{odd}}$, where $\mathbb B^{n-k,k}_{\mathrm{even}}$ (resp. $\mathbb B^{n-k,k}_{\mathrm{odd}}$) consists of all cup diagrams with an even (resp. odd) number of dots.

Let $\mathbb S^2\subseteq\mathbb R^3$ be the standard unit sphere on which we fix the points $p=(0,0,1)$ and $q=(1,0,0)$. Given a cup diagram $\ba\in\mathbb B^{n-k,k}$, we define $S_\ba\subseteq\left(\mathbb S^2\right)^m$ as the submanifold consisting of all $(x_1,\ldots,x_m)\in\left(\mathbb S^2\right)^m$ which satisfy the relations $x_i=-x_j$ (resp.\ $x_i=x_j$) if the vertices $i$ and $j$ are connected by an undotted cup (resp.\ dotted cup). Moreover, we impose the relations $x_i=p$ if the vertex $i$ is connected to a dotted ray and $x_i=-p$ (resp.\ $x_i=q$) if $i$ is connected to an undotted ray which is the rightmost ray in $\ba$ (resp.\ not the rightmost ray). The {\it topological Springer fiber $\mathcal S^{n-k,k}_D$ of type $D$} is defined as the union
\[
\mathcal S^{n-k,k}_D:=\bigcup_{\ba\in\mathbb B^{n-k,k}}S_\ba \subseteq\left(\mathbb S^2\right)^m.
\]
The above definition generalizes the construction of the topological Springer fiber in \cite[\S4.1]{ES12} from the equal row case to the general two row case and returns (up to a sign convention) the definition in the equal row case (cf.\ Remark \ref{rem:relation_to_ES_model} for the precise relationship).      

The first main result (cf.\ Theorem~\ref{thm:main_result_1}) of this article proves a conjecture by Ehrig and Stroppel \cite[Conjecture C]{ES12} on the topology of Springer fibers of type $D$ corresponding to partitions with two equal parts and at the same time extends the result to all two-row Springer fibers.  
\begin{introthm} \label{introthm_1}
There exists a homeomorphism $\mathcal S^{n-k,k}_D\cong\mathcal Fl^{n-k,k}_D$ such that the images of the $S_\ba$ are irreducible components of $\mathcal Fl^{n-k,k}_D$ for all $\ba\in\mathbb B^{n-k,k}$.
\end{introthm} 

The Springer fiber $\mathcal Fl^{n-k,k}_D$ decomposes into two connected components. Under the inverse of the homeomorphism in Theorem \ref{introthm_1} the two connected components of $\mathcal Fl^{n-k,k}_D$ are mapped onto $\mathcal S^{n-k,k}_{D,\mathrm{odd}}:=\bigcup_{\ba\in\mathbb B^{n-k,k}_\mathrm{odd}}S_\ba$ and $\mathcal S^{n-k,k}_{D,\mathrm{even}}:=\bigcup_{\ba\in\mathbb B^{n-k,k}_\mathrm{even}}S_\ba$, respectively. Let $\mathcal Fl^{n-k,k}_{D,\mathrm{odd}}$ denote the image of $\mathcal S^{n-k,k}_{D,\mathrm{odd}}$ under the homeomorphism $\mathcal S^{n-k,k}_D\cong\mathcal Fl^{n-k,k}_D$, i.e.\ it is one of the connected components of $\mathcal Fl^{n-k,k}_D$. The following result relates the two-row Springer fibers of type $C$ and $D$:

\begin{introthm} \label{introthm_2}
There exists an isomorphism of algebraic varieties $\mathcal Fl^{n-k,k}_{D,\mathrm{odd}}\cong\mathcal Fl^{n-k-1,k-1}_C$, i.e.\ the $(n-k-1,k-1)$ Springer fiber of type $C$ is isomorphic (as an algebraic variety) to one of the connected components of the $(n-k,k)$ Springer fiber of type $D$, which can be written down explicitly. In particular, the topological model of the type $D$ Springer fiber also provides a topological model for the type $C$ Springer fiber. More precisely, we have a homeomorphism $\mathcal S^{n-k,k}_{D,\mathrm{odd}}\cong\mathcal Fl^{n-k-1,k-1}_C$.
\end{introthm}

In \cite{HL14} two-row Slodowy slices of type $C$ and $D$ were studied via fixed-point subvarieties of certain Nakajima quiver varieties arising from diagram automorphisms. The authors show that the Slodowy slice of type $D$ to the orbit with Jordan type $(n-k,k)$ is isomorphic to the Slodowy slice of type $C$ to the orbit with Jordan type $(n-k-1,k-1)$. They also ask whether there is an isomorphism between the resolutions of these singular affine varieties or an isomorphism between the corresponding Springer fibers (cf.\ \cite[1.3]{HL14}). Theorem \ref{introthm_2} provides an affirmative answer to the latter question.

Since nilpotent orbits of the odd orthogonal group (type $B$) are parameterized by partitions in which even parts occur with even multiplicity it follows that there cannot exist a two-row partition labelling a nilpotent orbit of type $B$ and hence there are no two-row Springer fibers of type $B$. The only interesting Springer fibers of type $B$ would correspond to nilpotent endomorphisms with at least three Jordan blocks (every one-row Springer fiber consists of a single flag only). Those are a lot more difficult to treat and so far there has been no significant progress (not even in type $A$) in determining their topology in an explicit and combinatorially satisfying way (at least there are no results comparable to those obtained in e.g.\ \cite{Fun03} or \cite{Rus11}).

\subsection*{Overview of the article}
In the following we discuss the contents of this article in more detail and sketch the ideas behind the proofs of the main theorems. 

In Section~\ref{sec:Algebraic_Springer_fibers} we recall basic definitions and facts about the (algebraic) Springer fibers of type $C$ and $D$ and review some known results concerning the combinatorics of the irreducible components. In particular, we review the parameterization of the irreducible components of the Springer fiber in terms of signed domino tableaux as introduced by van Leeuwen in his thesis \cite{vL89} based on earlier work by Spaltenstein~\cite{Spa82} (see also \cite{Pie04}). 

In Section~\ref{sec:Topological_Springer_fibers} the topological Springer fibers are defined and in Proposition \ref{prop:top_intersections} we provide a combinatorial description of the topology of the pairwise intersections of the submanifolds $S_\ba\subseteq\left(\mathbb S^2\right)^m$, $\ba\in\mathbb B^{n-k,k}$, which yields (in connection with Theorem~\ref{introthm_1}) a combinatorial description of the topology of intersections of the irreducible components of $\mathcal Fl^{n-k,k}_D$. In contrast to type $A$, the intersections of the Springer fiber $\mathcal Fl^{n-k,k}_D$ cannot be described using the highest weight Lie theory combinatorics \cite{LS13,ES13} if $m\neq k$ (see also the discussion in \cite[\S6.5]{ES12}). In particular, the convolution algebras arising from the Springer fibers $\mathcal Fl^{n-k,k}_D$ (by mimicking the approach of \cite{SW12}) are in general not isomorphic to the corresponding Khovanov algebras of type $D$ constructed in \cite{ES13} (this is true if and only if $k=m$).

Sections~\ref{sec:embedded_Springer_fiber} and \ref{sec:topology_irred_comp} are concerned with the proof of the main theorems. The main idea is the following:

Let $N>0$ be a large integer and let $z\colon\mathbb C^{2N} \to \mathbb C^{2N}$ be a nilpotent linear operator with two equally-sized Jordan blocks. In \cite[\S2]{CK08} the authors define a smooth projective variety 
\[
Y_m=\big\{\left(F_1,\ldots,F_m\right) \mid F_i \subseteq\mathbb C^{2N} \textrm{ has dimension }i,\, F_1\subseteq\ldots\subseteq F_m,\, zF_i\subseteq F_{i-1}\big\}
\]
and construct an explicit diffeomorphism $\phi_m\colon Y_m\xrightarrow\cong \left(\mathbb P^1\right)^m$. The variety $Y_m$ should be seen as a compactification of the preimage of a Slodowy slice of type $A$ under the Springer resolution. The diffeomorphism $\phi_m$ also plays a crucial role in establishing topological models for the Springer fibers of type $A$ which are naturally embedded in $Y_m$ (cf.\ \cite{Weh09} and \cite{Rus11}). It turns out that the two-row Springer fibers of type $D$, resp.\ of type $C$, can also be embedded into $Y_m$, resp.\ $Y_{m-1}$ (see Section~\ref{sec:embedded_Springer_fiber}).

Furthermore, we introduce a diffeomorphism $\gamma_{n-k,k}\colon\left(\mathbb S^2\right)^m\to\left(\mathbb P^1\right)^m$ (unlike the diffeomorphism $\phi_m$ this diffeomorphism actually depends on the partition). This diffeomorphism does not play a vital role and it is only introduced for cosmetic reasons. 

In order to prove Theorem~\ref{introthm_1} one needs to check that the image of $\mathcal S^{n-k,k}_D\subseteq\left(\mathbb S^2\right)^m$ under the diffeomorphism $\phi_m^{-1}\circ\gamma_{n-k,k}$ is the embedded Springer fiber $\mathcal Fl^{n-k,k}_D\subseteq Y_m$. In Lemma \ref{lem:sphere_vs_projective_space} we provide the first step by giving an explicit description of the image of $\mathcal S^{n-k,k}_D$ under the map $\gamma_{n-k,k}$. The following picture summarizes the results and constructions discussed so far:
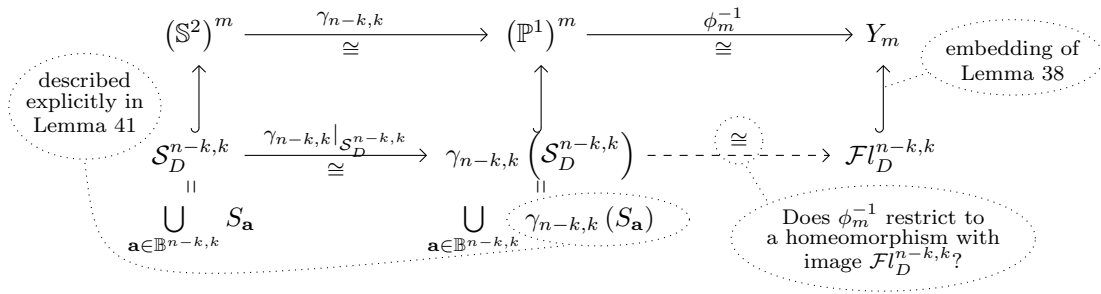
\[
\begin{tikzpicture}[>=angle 90]
\node at (0,0) {$\left(\mathbb S^2\right)^m$};
\node at (4.5,0) {$\left(\mathbb P^1\right)^m$};
\node at (9,0) {$Y_m$};
\node at (-.1,-1.6) {$\mathcal S^{n-k,k}_D$};
\node at (-.1,-2) {$\shortparallel$};
\node at (-.1,-2.6) {$\bigcup\limits_{\ba\in\mathbb B^{n-k,k}}S_\ba$};
\node at (4.5,-1.6) {$\gamma_{n-k,k}\left(\mathcal S^{n-k,k}_D\right)$};
\node at (4.5,-2) {$\shortparallel$};
\node at (4.5,-2.6) {$\bigcup\limits_{\ba\in\mathbb B^{n-k,k}}\gamma_{n-k,k}\left(S_\ba\right)$};
\draw[dotted] (5.25,-2.45) ellipse (1.2cm and .35cm);
\node at (9.1,-1.6) {$\mathcal Fl^{n-k,k}_D$};

\draw[dotted] (9,-2.7) ellipse (2cm and .7cm);
\node at (9,-2.4) {\footnotesize Does $\phi_m^{-1}$ restrict to};
\node at (9,-2.7) {\footnotesize a homeomorphism with};
\node at (9,-3) {\footnotesize image $\mathcal Fl^{n-k,k}_D$?};
\draw[dotted] (7.2,-1.75) -- (7.5,-2.25);

\draw[dotted] (10.7,-.3) ellipse (1.2cm and .5cm);
\node at (10.7,-.2) {\footnotesize embedding of};
\node at (10.7,-.5) {\footnotesize Lemma \ref{lem:embedded_Springer_fiber}};
\draw[dotted] (9.6,-.5) -- (9,-.7);

\draw[dotted] (-1.5,-.95) ellipse (1cm and .7cm);
\node at (-1.5,-.6) {\footnotesize described};
\node at (-1.5,-.9) {\footnotesize explicitly in};
\node at (-1.5,-1.2) {\footnotesize Lemma \ref{lem:sphere_vs_projective_space}};
\draw[dotted] (-1.5,-1.7) .. controls (-1.3,-3.4) .. (4.8,-2.8);

\draw[->] (.6,0) -- (3.8,0);
\node at (2,.2) {\footnotesize $\gamma_{n-k,k}$};
\node at (2,-.2) {\footnotesize $\cong$};

\draw[->] (.6,-1.6) -- (3,-1.6);
\node at (1.8,-1.4) {\footnotesize $\gamma_{n-k,k}\vert_{\mathcal S^{n-k,k}_D}$};
\node at (1.8,-1.8) {\footnotesize $\cong$};

\draw[->] (5.1,0) -- (8.6,0);
\node at (6.9,.2) {\footnotesize $\phi_m^{-1}$};
\node at (6.9,-.2) {\footnotesize $\cong$};

\draw[->,dashed] (5.9,-1.6) -- (8.3,-1.6);
\node at (7.1,-1.4) [circle,draw,dotted] {\footnotesize $\cong$};

\draw[right hook->] (0,-1.3) -- (0,-.4);
\draw[right hook->] (4.5,-1.3) -- (4.5,-.4);
\draw[right hook->] (9,-1.3) -- (9,-.4);
\end{tikzpicture}
\]
In Section~\ref{sec:topology_irred_comp} we answer the remaining question in the picture above by showing that $\phi_m^{-1}$ does indeed restrict to a homeomorphism $\gamma_{n-k,k}\left(\mathcal S^{n-k,k}_D\right)\cong\mathcal Fl^{n-k,k}_D$. Note that it suffices to prove the following statement (cf. Proposition \ref{prop:preimage_contained}):
\begin{introprop}
The preimages of the sets $\gamma_{n-k,k}(S_\ba)$ under $\phi_m$ are pairwise different irreducible components of $\mathcal Fl^{n-k,k}_D\subseteq Y_m$ for all $\ba\in\mathbb B^{n-k,k}$. 
\end{introprop}
Since the irreducible components of $\mathcal Fl^{n-k,k}_D$ are in bijective correspondence with cup diagrams in $\mathbb B^{n-k,k}$ we deduce that the inclusion 
\[
\phi_m^{-1}\left(\gamma_{n-k,k}\left(\mathcal S^{n-k,k}_D\right)\right)=\bigcup_{\ba\in\mathbb B^{n-k,k}} \phi_m^{-1}\left(\gamma_{n-k,k}\left(S_\ba\right)\right)\subseteq\mathcal Fl^{n-k,k}_D
\]
is in fact an equality which finishes the proof of Theorem~\ref{introthm_1}. 

In order to prove the above proposition we proceed by induction on the number of undotted cups in $\ba\in\mathbb B^{n-k,k}$ which is is more or less the same proof as in type $A$ (cf.\ \cite{Rus11}). One only needs to be careful about the additional isotropy condition (cf. Lemma~\ref{z_respects_isotropy_lem}). Thus, the main difficulty lies in establishing the induction start, i.e.\ to prove the claim for cup diagrams without any undotted cups. This is done in Proposition~\ref{prop:preimage_no_undotted_cups} (which itself is a proof by induction on the number of dotted cups) and is considered the technical heart of the argument because it requires new techniques which are not straightforward generalizations of the type $A$ case. 

In order to prove Theorem~\ref{introthm_2} we consider the surjective morphism of varieties $\pi_m\colon Y_m\twoheadrightarrow Y_{m-1}$ given by $(F_1,\ldots,F_m)\mapsto(F_1,\ldots,F_{m-1})$ and show (using similar arguments as in the proof of Theorem~\ref{introthm_1}) that the restriction of $\pi_m$ to $\mathcal Fl^{n-k,k}_{D,\mathrm{odd}}\subseteq Y_m$ yields a homeomorphism (even an isomorphism of varieties) whose image is the embedded Springer fiber $\mathcal Fl^{n-k-1,k-1}_C\subseteq Y_{m-1}$. Deleting the vector space $F_m$ of a given flag $(F_1,\ldots,F_m)\in\mathcal Fl^{n-k,k}_{D,\mathrm{odd}}$ in order to pass from type $D$ to type $C$ fits nicely into the combinatorial picture since signed domino tableaux of type $C$ can be obtained from signed domino tableaux of type $D$ by deleting the domino labelled $m$ (cf.\ Lemma~\ref{lem:bijection_tableaux_C_D} for details). 

\subsection*{Acknowledgements}
This article is part of the author's PhD thesis. The author would like to thank his advisor Catharina Stroppel for many interesting and useful discussions. 

\section{Algebraic Springer fibers} \label{sec:Algebraic_Springer_fibers}

We begin by defining the (algebraic) Springer fibers and provide an overview over some known results concerning the combinatorics of the irreducible components. The purpose is to set up notations and establish conventions used throughout this article which sometimes differ slightly from the ones in related publications. Unless stated otherwise $n=2m$ denotes an even positive integer.

\subsection{Nilpotent orbits and algebraic Springer fibers}

Let $\mathcal N\subseteq\mathfrak{sl}(\mathbb C^n)$ be the {\it nilpotent cone} consisting of all nilpotent endomorphisms of $\mathbb C^n$ (in the usual sense of linear algebra). The Jordan normal form implies that the orbits under the conjugation action of the special linear group $\mathrm{SL}(\mathbb C^n)$ on $\mathcal N$ can be parameterized in terms of {\it partitions} of $n$, i.e.\ $r$-tuples $\lambda=(\lambda_1,\ldots,\lambda_r)\in\mathbb Z^r_{>0}$, $r\in\mathbb Z_{>0}$, of positive integers such that $\lambda_1\geq\lambda_2\geq\ldots\geq\lambda_r$ and $\lambda_1+\ldots+\lambda_r=n$, where the {\it parts} $\lambda_i$ of $\lambda$ encode the sizes of the Jordan blocks of the elements contained in an orbit. Let $\mathcal P(n)$ denote the set of all partitions of $n$. 

Fix an element $\epsilon\in\{\pm 1\}$ and let $\beta_\epsilon$ be a nondegenerate bilinear form on $\mathbb C^n$ which satisfies $\beta_\epsilon(v,w)=\epsilon(v,w)$ for all $v,w\in\mathbb C^n$. Let $\mathrm{Aut}(\mathbb C^n,\beta_\epsilon)$ denote the isometry group consisting of all linear automorphisms of $\mathbb C^n$ preserving $\beta_\epsilon$. The Lie algebra $\mathfrak{aut}(\mathbb C^n,\beta_\epsilon)$ of $\mathrm{Aut}(\mathbb C^n,\beta_\epsilon)$ is the subalgebra of $\mathfrak{sl}(\mathbb C^n)$ consisting of all endomorphisms $x$ of $\mathbb C^n$ which satisfy the equation $\beta_\epsilon(x(v),w)=-\beta_\epsilon(v,x(w))$ for all $v,w\in\mathbb C^n$. Note that $\mathrm{Aut}(\mathbb C^n,\beta_1)\cong O_n(\mathbb C)$ and $\mathfrak{aut}(\mathbb C^n,\beta_{-1})\cong\mathfrak{so}_n(\mathbb C)$ if $\beta_\epsilon$ is nondegenerate and symmetric, whereas $\mathrm{Aut}(\mathbb C^n,\beta_{-1})\cong Sp_n(\mathbb C)$ and $\mathfrak{aut}(\mathbb C^n,\beta_{-1})\cong\mathfrak{sp}_n(\mathbb C)$ if $\beta$ is symplectic.

We define $\mathcal P_\epsilon(n)$ as the subset of $\mathcal P(n)$ consisting of all partitions $\lambda$ of $n$ for which the cardinality of $\{i\mid\lambda_i=j\}$ is even for all $j$ satisfying $(-1)^j=\epsilon$, i.e.\ even (resp.\ odd) parts occur with even multiplicity. We refer to the partitions in $\mathcal P_1(n)$ (resp.\ $\mathcal P_{-1}(n)$) as {\it admissible of type} $D$ (resp.\ type $C$) since they parameterize nilpotent orbits in the simple Lie algebra of the respective type. The following classification of nilpotent orbits is well known \cite{Wil37,Ger61}.

\begin{prop} \label{prop:Gerstenhaber}
The orbits under the conjugation-action of $\mathrm{Aut}(\mathbb C^n,\beta_\epsilon)$ on the variety of nilpotent elements $\mathcal N_{\mathfrak{aut}(\mathbb C^n,\beta_\epsilon)}=\mathcal N\cap\mathfrak{aut}(\mathbb C^n,\beta_\epsilon)$ are in bijective correspondence with the partitions contained in $\mathcal P_\epsilon(n)$. The parts of the partition associated with the orbit of an endomorphism encode the sizes of the Jordan blocks in Jordan normal form.
\end{prop}

\begin{defi} 
A {\it full isotropic flag} in $\mathbb C^n$ (with respect to $\beta_\epsilon$) is a sequence $F_\bullet$ of subspaces $\{0\}=F_0\subsetneq F_1\subsetneq\ldots\subsetneq F_m$ of $\mathbb C^n$ such that $F_m$ is isotropic with respect to $\beta_\epsilon$, i.e.\ $\beta_\epsilon$ vanishes on $F_m\times F_m$. The set of all full isotropic flags is denoted by $\mathcal Fl_{\beta_\epsilon}$.
\end{defi}

Since the inclusions of the subspaces of a flag $F_\bullet$ are strict, $F_m$ is maximal isotropic and we have $\dim(F_i)=i$ for all $i\in\{1,\ldots,m\}$. The set $\mathcal Fl_{\beta_\epsilon}$ can be equipped with the structure of a smooth projective variety, e.g.\ by identifying it with a homogeneous $\mathrm{Aut}(\mathbb C^n,\beta_\epsilon)$-space. Adding the vector spaces $F_{n-i}=F_i^\perp$ to a given full isotropic flag $F_\bullet$ (the orthogonal complement is taken with respect to $\beta_\epsilon$) defines an embedding of $\mathcal Fl_{\beta_\epsilon}$ into the {\it full flag variety $\mathcal Fl$ of type} $A$. Given any other nondegenerate symmetric (resp.\ symplectic) bilinear form $\beta$ on $\mathbb C^n$, the corresponding varieties $\mathcal Fl_\beta$ and $\mathcal Fl_{\beta_1}$ (resp.\ $\mathcal Fl_{\beta_{-1}}$) are isomorphic which allows us to speak about the {\it full flag variety of type} $D$ (resp.\ {\it type} $C$), denoted by $\mathcal Fl_D$ (resp.\ $\mathcal Fl_C$), without further specifying a nondegenerate symmetric (resp.\ symplectic) bilinear form.   

\begin{rem} \label{rem:companion_flag}
According to our conventions the full flag variety $\mathcal Fl_D$ of type $D$ is isomorphic to a quotient of $O_n(\mathbb C)$ (and not $SO_n(\mathbb C)$). Hence, it consists of two isomorphic connected components. The component containing a given flag $F_\bullet$ is determined by $F_m$. More precisely, there is a unique flag $F_\bullet'$ such that $F_i=F_i'$ for all $i\in\{1,\ldots,m-1\}$ and $F_m\neq F_m'$ and the two flags lie in different connected components (cf.\ \cite[\S1.4]{vL89} or \cite[Remark 2.2]{ES12}). 
\end{rem}

\begin{defi} \label{defi_springer_fiber}
The {\it (algebraic) Springer fiber} $\mathcal Fl_{\beta_\epsilon}^{x}$ associated with $\beta_\epsilon$ and $x \in \mathcal N_{\mathfrak{g}_\epsilon}$ is the projective subvariety of $\mathcal Fl_{\beta_\epsilon}$ consisting of all isotropic flags $F_\bullet$ which satisfy the conditions $xF_i \subseteq F_{i-1}$ for all $i\in\{1,\ldots,m\}$.
\end{defi}

If $\beta$ is another nondegenerate symmetric (resp.\ symplectic) bilinear form on $\mathbb C^n$ and $y$ a nilpotent endomorphism of $\mathbb C^n$ contained in $\mathfrak{aut}(\mathbb C^n,\beta)$, then the Springer fibers $\mathcal Fl^y_\beta$ and $\mathcal Fl^x_{\beta_\epsilon}$ are isomorphic if and only if $x$ and $y$ have the same Jordan type.  
Thus, by Proposition~\ref{prop:Gerstenhaber}, there is (up to isomorphism) precisely one Springer fiber for every admissible partition $\lambda$ of type $D$ (resp.\ type $C$) which we denote by $\mathcal Fl_D^\lambda$ (resp.\ $\mathcal Fl_C^\lambda$) assuming that some nondegenerate symmetric (resp.\ symplectic) bilinear form and a compatible nilpotent endomorphism of Jordan type $\lambda$ have been fixed beforehand. 


\subsection{Irreducible components and combinatorics} \label{sec:irred_comp}

Recall that a partition $\lambda=(\lambda_1,\ldots,\lambda_r)$ of $n$ can be depicted as a Young diagram, i.e.\ a collection of $n$ boxes arranged in $r$ left-aligned rows, where the $i$-th row consists of $\lambda_i$ boxes (this is commonly known as English notation). 

\begin{defi}
A {\it standard Young tableau of shape $\lambda$} is a filling of the Young diagram of $\lambda\in\mathcal P(n)$ with the numbers $1,2,\ldots,n$ such that each number occurs exactly once and the entries decrease in every row and column. Let $SYT(\lambda)$ be the set of all standard Young tableaux of shape $\lambda$.
\end{defi} 

\begin{ex} \label{ex:std_yng_tab}
Here is a complete list of all elements contained in the set $SYT(3,2)$:
\[
\young(543,21) \hspace{3em} \young(542,31) \hspace{3em} \young(532,41) \hspace{3em} \young(531,42) \hspace{3em} \young(541,32)
\]
\end{ex}

Let $x\in\mathcal N$ be a nilpotent endomorphism of $\mathbb C^n$ of Jordan type $\lambda$ and let $\mathcal Fl^x$ be the associated Springer fiber (of type $A$) consisting of all sequences $\{0\}=F_0\subsetneq F_1 \subsetneq\ldots\subsetneq F_n=\mathbb C^n$ of subspaces satisfying $xF_i\subseteq F_{i-1}$ for all $i\in\{1,\ldots,n\}$. By work of Spaltenstein \cite{Spa76} and Vargas \cite{Var79} there exists a surjection $\mathcal S^x\colon \mathcal Fl^x \twoheadrightarrow SYT(\lambda)$ which can be used to parameterize the irreducible components of the Springer fiber as follows:

\begin{prop} \label{prop:Spaltenstein-Vargas}
The set $SYT(\lambda)$ of standard Young tableaux of shape $\lambda$ is in bijective correspondence with the irreducible components of the Springer fiber $\mathcal Fl^x$ via the map which sends a tableau $T$ to the closure of the fiber of $\mathcal S^x$ over $T$.  
\end{prop}

In order to obtain a parameterization of the irreducible components of the Springer fibers of type $D$ and $C$ one has to replace standard Young tableaux by admissible domino tableaux. For the rest of this section we assume that $n=2m$ is even.    

\begin{defi} \label{domino_diag_defi} 
 A {\it domino diagram of shape $\lambda$} with $m$ dominoes is given by a partitioning of the set of boxes of the Young diagram corresponding to $\lambda$ into two-element subsets, called {\it dominoes}, such that any two associated boxes have a common vertical or horizontal edge. 
\end{defi}

We depict a domino diagram by deleting the common edge of any two boxes forming a domino.

\begin{ex}
Here is a complete list of all domino diagrams of shape $(5,3)$:
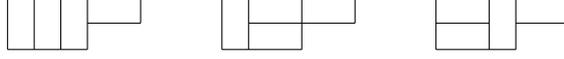
\[
\begin{tikzpicture}[scale=0.7]

\draw (0,0) -- +(0,1);
\draw (.5,0) -- +(0,1);
\draw (0,0) -- +(1.5,0);
\draw (0,1) -- +(2.5,0);
\draw (1.5,0) -- +(0,1);
\draw (1,0) -- +(0,1);
\draw (2.5,.5) -- +(0,.5);
\draw (1.5,.5) -- +(1,0);

\end{tikzpicture}
\hspace{3em}
\begin{tikzpicture}[scale=0.7]

\draw (0,0) -- +(0,1);
\draw (.5,0) -- +(0,1);
\draw (0,0) -- +(1.5,0);
\draw (0,1) -- +(2.5,0);
\draw (.5,.5) -- +(2,0);
\draw (1.5,0) -- +(0,1);
\draw (2.5,.5) -- +(0,.5);

\end{tikzpicture}
\hspace{3em}
\begin{tikzpicture}[scale=0.7]

\draw (0,0) -- +(0,1);
\draw (0,0) -- +(1.5,0);
\draw (0,1) -- +(2.5,0);
\draw (0,.5) -- +(1,0);
\draw (1.5,0) -- +(0,1);
\draw (1,0) -- +(0,1);
\draw (2.5,.5) -- +(0,.5);
\draw (1.5,.5) -- +(1,0);

\end{tikzpicture}
\]
\end{ex}

\begin{defi}
An {\it admissible domino tableau} of shape $\lambda$, where $\lambda\in\mathcal P_\epsilon(n)$, is obtained from a domino diagram of shape $\lambda$ by filling the $n$ boxes of its underlying Young diagram with the numbers $1,\ldots,m$ such that: 
\begin{enumerate}[{(\text{ADT}}1{)}]
\item Each of the numbers $1,\ldots,m$ occurs exactly twice and any two boxes with the same number form a domino.
\item The entries in each row and column are weakly decreasing from left to right and top to bottom.
\item The partition corresponding to the shape of the diagram obtained by deleting the dominoes labelled with $1,\ldots,i$ is admissible of type $D$ (resp.\ type $C$) for every $i\in\{1,\ldots,m\}$ if $\epsilon=1$ (resp.\ $\epsilon=-1$). 
\end{enumerate}
The set of all admissible domino tableaux of shape $\lambda$ is denoted by $ADT(\lambda)$.
\end{defi}

When depicting an admissible domino tableau we draw dominoes instead of paired boxes and only write a single number in every domino.

\begin{ex} \label{ex:admissible_domino_tableaux}
Here is a complete list of all admissible domino tableaux of shape $(5,3)$:
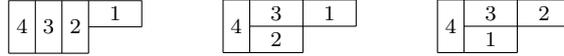
\[
\begin{tikzpicture}[scale=0.7]

\draw (0,0) -- +(0,1);
\draw (.5,0) -- +(0,1);
\draw (0,0) -- +(1.5,0);
\draw (0,1) -- +(2.5,0);
\draw (1.5,0) -- +(0,1);
\draw (1,0) -- +(0,1);
\draw (2.5,.5) -- +(0,.5);
\draw (1.5,.5) -- +(1,0);

\begin{footnotesize}
\node at (.25,.5) {4};
\node at (.75,.5) {3};
\node at (1.25,.5) {2};
\node at (2,.75) {1};
\end{footnotesize}

\end{tikzpicture}
\hspace{3em}
\begin{tikzpicture}[scale=0.7]

\draw (0,0) -- +(0,1);
\draw (.5,0) -- +(0,1);
\draw (0,0) -- +(1.5,0);
\draw (0,1) -- +(2.5,0);
\draw (.5,.5) -- +(2,0);
\draw (1.5,0) -- +(0,1);
\draw (2.5,.5) -- +(0,.5);

\begin{footnotesize}
\node at (.25,.5) {4};
\node at (1,.75) {3};
\node at (1,.25) {2};
\node at (2,.75) {1};
\end{footnotesize}

\end{tikzpicture}
\hspace{3em}
\begin{tikzpicture}[scale=0.7]

\draw (0,0) -- +(0,1);
\draw (.5,0) -- +(0,1);
\draw (0,0) -- +(1.5,0);
\draw (0,1) -- +(2.5,0);
\draw (.5,.5) -- +(2,0);
\draw (1.5,0) -- +(0,1);
\draw (2.5,.5) -- +(0,.5);

\begin{footnotesize}
\node at (.25,.5) {4};
\node at (1,.75) {3};
\node at (1,.25) {1};
\node at (2,.75) {2};
\end{footnotesize}

\end{tikzpicture}.
\]
Furthermore, here are all admissible domino tableaux of shape $(4,2)$:
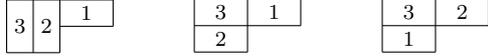
\[
\begin{tikzpicture}[scale=0.7]

\draw (0,0) -- +(0,1);
\draw (.5,0) -- +(0,1);
\draw (0,0) -- +(1,0);
\draw (0,1) -- +(2,0);
\draw (1,0) -- +(0,1);
\draw (2,.5) -- +(0,.5);
\draw (1,.5) -- +(1,0);

\begin{footnotesize}
\node at (.25,.5) {3};
\node at (.75,.5) {2};
\node at (1.5,.75) {1};
\end{footnotesize}

\end{tikzpicture}
\hspace{3em}
\begin{tikzpicture}[scale=0.7]

\draw (0,0) -- +(0,1);
\draw (0,0) -- +(1,0);
\draw (0,1) -- +(2,0);
\draw (0,.5) -- +(2,0);
\draw (1,0) -- +(0,1);
\draw (2,.5) -- +(0,.5);

\begin{footnotesize}
\node at (.5,.75) {3};
\node at (.5,.25) {2};
\node at (1.5,.75) {1};
\end{footnotesize}

\end{tikzpicture}
\hspace{3em}
\begin{tikzpicture}[scale=0.7]

\draw (0,0) -- +(0,1);
\draw (0,0) -- +(1,0);
\draw (0,1) -- +(2,0);
\draw (0,.5) -- +(2,0);
\draw (1,0) -- +(0,1);
\draw (2,.5) -- +(0,.5);

\begin{footnotesize}
\node at (.5,.75) {3};
\node at (.5,.25) {1};
\node at (1.5,.75) {2};
\end{footnotesize}

\end{tikzpicture}.
\]
\end{ex}




Let $x\in\mathcal N_{\mathfrak{g}_\epsilon}$ be a nilpotent endomorphism of $\mathbb C^n$ of Jordan type $\lambda$. We define a surjection $\mathcal S^x_{\beta_\epsilon}\colon \mathcal Fl^x_{\beta_\epsilon}\twoheadrightarrow ADT(\lambda)$ as follows: Given $F_\bullet\in\mathcal Fl^x_{\beta_\epsilon}$, we obtain a sequence $x^{(m)},\ldots,x^{(0)}$ of nilpotent endomorphisms, where $x^{(i)}\colon F_i^\perp/F_i\to F_i^\perp/F_i$ denotes the map induced by $x$ (the orthogonal complement is taken with respect to $\beta_\epsilon$). The Jordan types $J(x^{(m-1)}),\ldots,J(x^{(i)}),\ldots,J(x^{(0)})$ are admissible partitions, where successive Jordan types differ by precisely one domino. We label the new domino by comparing Jordan types of $x^{(i)}$ and $x^{(i+1)}$ with $i$. More details can be found in \cite{Spa82}, \cite{vL89} or \cite{Pie04}.

\begin{ex} \label{ex:Spaltenstein_map}
Let $x\in\mathfrak{sl}(\mathbb C^8)$ be a nilpotent endomorphism of Jordan type $(5,3)$ and let
\begin{equation} \label{eq:5_3_Jordan_basis}
\begin{tikzpicture}[baseline={(0,0)}]
\node (e1) at (0,0) {$e_1$};
\node (e2) at (1,0) {$e_2$};
\node (e3) at (2,0) {$e_3$};
\node (e4) at (3,0) {$e_4$};
\node (e5) at (4,0) {$e_5$};
\node (f1) at (5.5,0) {$f_1$};
\node (f2) at (6.5,0) {$f_2$};
\node (f3) at (7.5,0) {$f_3$};
\path[->,font=\scriptsize,>=angle 90,bend right]
(e2) edge (e1)
(e3) edge (e2)
(e4) edge (e3)
(e5) edge (e4)
(f2) edge (f1)
(f3) edge (f2);
\end{tikzpicture}
\end{equation}
be a Jordan basis (the arrows indicate the action of $x$, $e_1$ and $f_1$ are mapped to zero by $x$). Let $\beta$ be the nondegenerate symmetric bilinear form on $\mathbb C^8$ given by 
\[
\beta(e_i,f_j)=0, \hspace{1.6em} \beta(e_i,e_{i'})=(-1)^{i-1}\delta_{i+i',n-k+1}, \hspace{1.6em} \beta(f_j,f_{j'})=(-1)^j\delta_{j+j',k+1},
\]
where $i,i'\in\{1,\ldots,5\}$ and $j,j'\in\{1,2,3\}$. Then one easily checks that $x\in\mathcal N_\mathfrak{g}$. Consider the flag $F_\bullet\in\mathcal Fl^{x}_\beta$, where
\[
F_1=\spn(e_1)\,,\,\, F_2=\spn(e_1,f_1)\,,\,\, F_3=\spn(e_1,e_2,f_1)\,,\,\, F_4=\spn(e_1,e_2,f_1,{\bf i}e_3+f_2).
\]
Note that $F_1^\perp$ is spanned by all vectors in (\ref{eq:5_3_Jordan_basis}) except $e_5$, $F_2^\perp$ is spanned by all vectors except $e_5,f_3$ and $F_3^\perp=\spn(e_1,e_2,e_3,f_1,f_2)$. Hence, one easily computes the Jordan types
\[
J(x^{(3)})=(1,1)\,,\,\,J(x^{(2)})=(3,1)\,,\,\,J(x^{(1)})=(3,3)\,,\,\,J(x^{(0)})=(5,3)
\]
and the tableau associated with $F_\bullet$ via $\mathcal S^x_\beta$ is constructed as follows:
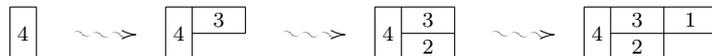
\[
\begin{tikzpicture}[scale=0.7, baseline={(0,.25)}]

\draw (0,0) -- +(0,1);
\draw (.5,0) -- +(0,1);
\draw (0,0) -- +(.5,0);
\draw (0,1) -- +(.5,0);

\begin{footnotesize}
\node at (.25,.5) {4};
\end{footnotesize}

\end{tikzpicture}
\hspace{.5em}
\begin{xy}
	\xymatrix@=1em{
\ar@{~>}[rr] &&
}
\end{xy}
\hspace{.5em}
\begin{tikzpicture}[scale=0.7, baseline={(0,.25)}]

\draw (0,0) -- +(0,1);
\draw (.5,0) -- +(0,1);
\draw (0,0) -- +(.5,0);
\draw (0,1) -- +(1.5,0);
\draw (.5,.5) -- +(1,0);
\draw (1.5,.5) -- +(0,.5);

\begin{footnotesize}
\node at (.25,.5) {4};
\node at (1,.75) {3};
\end{footnotesize}

\end{tikzpicture}
\hspace{.5em}
\begin{xy}
	\xymatrix@=1em{
\ar@{~>}[rr] &&
}
\end{xy}
\hspace{.5em}
\begin{tikzpicture}[scale=0.7, baseline={(0,.25)}]

\draw (0,0) -- +(0,1);
\draw (.5,0) -- +(0,1);
\draw (0,0) -- +(1.5,0);
\draw (0,1) -- +(1.5,0);
\draw (.5,.5) -- +(1,0);
\draw (1.5,0) -- +(0,1);

\begin{footnotesize}
\node at (.25,.5) {4};
\node at (1,.75) {3};
\node at (1,.25) {2};
\end{footnotesize}

\end{tikzpicture}
\hspace{.5em}
\begin{xy}
	\xymatrix@=1em{
\ar@{~>}[rr] &&
}
\end{xy}
\hspace{.5em}
\begin{tikzpicture}[scale=0.7, baseline={(0,.25)}]

\draw (0,0) -- +(0,1);
\draw (.5,0) -- +(0,1);
\draw (0,0) -- +(1.5,0);
\draw (0,1) -- +(2.5,0);
\draw (.5,.5) -- +(2,0);
\draw (1.5,0) -- +(0,1);
\draw (2.5,.5) -- +(0,.5);

\begin{footnotesize}
\node at (.25,.5) {4};
\node at (1,.75) {3};
\node at (1,.25) {2};
\node at (2,.75) {1};
\end{footnotesize}

\end{tikzpicture}
\]
\end{ex}

\begin{rem} \label{rem:irreducible_and_connected_components}
The closures of the preimages of the elements in $ADT(\lambda)$ under the map $\mathcal S^x_{\beta_\epsilon}$ are not necessarily connected and hence cannot be irreducible (as in type $A$). Nonetheless, the irreducible components of $\mathcal Fl^x_{\beta_\epsilon}$ contained in the closure of $(\mathcal S^x_{\beta_\epsilon})^{-1}(T)$ are precisely its connected components \cite[Lemma 3.2.3]{vL89}.
\end{rem}

\begin{ex} \label{ex:flags_in_irred_comp}
Consider the Springer fiber $\mathcal Fl^{x}_\beta$ as in Example~\ref{ex:Spaltenstein_map}. Then one can check (e.g.\ by using the inductive construction in \cite[\S6.5]{ES12}) that the fiber of $\mathcal S^x_\beta$ over $T$, where $T$ denotes the rightmost admissible domino tableau of shape $(5,3)$ in Example~\ref{ex:admissible_domino_tableaux}, is the union of the following two disjoint sets of flags
\[
\big\{\spn(\mu e_1+f_1)\subseteq\spn(e_1,f_1)\subseteq\spn(e_1,e_2,f_1)\subseteq\spn(e_1,e_2,f_1,{\bf i}e_3+f_2)\mid\mu\in\mathbb C\big\}
\]
\[
\big\{\spn(\mu e_1+f_1)\subseteq\spn(e_1,f_1)\subseteq\spn(e_1,e_2,f_1)\subseteq\spn(e_1,e_2,f_1,{\bf i}e_3-f_2)\mid\mu\in\mathbb C\big\}
\]
each of which is isomorphic to an affine space $\mathbb A^1$. By taking the closure we add the possibility of choosing $\spn(e_1)$ as a one-dimensional subspace of these flags. Thus, the closure of the preimage of $T$ under $\mathcal S^x_\beta$ is isomorphic to a disjoint union of two projective spaces $\mathbb P^1$ each of which corresponds to an irreducible component of $\mathcal Fl^{x}_\beta$.  
\end{ex}

By the above discussion the map $\mathcal S^x_{\beta_\epsilon}$ does not provide a parameterization of the irreducible components of $\mathcal Fl^x_{\beta_\epsilon}$. It is necessary to add more combinatorial data to the tableaux in order to count the connected components.  

\begin{defi} A {\it signed domino tableau} of shape $\lambda$ is an admissible domino tableau of shape $\lambda$ together with a choice of sign (an element of the set $\{+,-\}$) for each vertical domino in an odd column if $\lambda\in\mathcal P_1(n)$ (resp.\ $\lambda\in\mathcal P_{-1}(n)$). We write $ADT^\text{sgn}(\lambda)$ for the set of signed domino tableau of shape $\lambda$. We define $ADT_\textrm{odd}^\text{sgn}(\lambda)$ (resp.\ $ADT_\textrm{even}^\text{sgn}(\lambda)$) as the set of all signed domino tableaux with an odd (resp.\ even) number of minus signs.
\end{defi}

\begin{ex} \label{ex:signed_domino_tableaux}
The set $ADT^\text{sgn}((5,3))$ consists of the following eight elements:
\[
\begin{tikzpicture}[scale=0.7]

\draw (0,0) -- +(0,1);
\draw (.5,0) -- +(0,1);
\draw (0,0) -- +(1.5,0);
\draw (0,1) -- +(2.5,0);
\draw (1.5,0) -- +(0,1);
\draw (1,0) -- +(0,1);
\draw (2.5,.5) -- +(0,.5);
\draw (1.5,.5) -- +(1,0);

\begin{footnotesize}
\node at (.25,.6) {4};
\node at (.25,.25) {$+$};
\node at (.75,.6) {3};
\node at (1.25,.6) {2};
\node at (1.25,.25) {$+$};
\node at (2,.75) {1};
\end{footnotesize}

\end{tikzpicture}
\hspace{3em}
\begin{tikzpicture}[scale=0.7]

\draw (0,0) -- +(0,1);
\draw (.5,0) -- +(0,1);
\draw (0,0) -- +(1.5,0);
\draw (0,1) -- +(2.5,0);
\draw (1.5,0) -- +(0,1);
\draw (1,0) -- +(0,1);
\draw (2.5,.5) -- +(0,.5);
\draw (1.5,.5) -- +(1,0);

\begin{footnotesize}
\node at (.25,.6) {4};
\node at (.25,.25) {$+$};
\node at (.75,.6) {3};
\node at (1.25,.6) {2};
\node at (1.25,.25) {$-$};
\node at (2,.75) {1};
\end{footnotesize}

\end{tikzpicture}
\hspace{3em}
\begin{tikzpicture}[scale=0.7]

\draw (0,0) -- +(0,1);
\draw (.5,0) -- +(0,1);
\draw (0,0) -- +(1.5,0);
\draw (0,1) -- +(2.5,0);
\draw (1.5,0) -- +(0,1);
\draw (1,0) -- +(0,1);
\draw (2.5,.5) -- +(0,.5);
\draw (1.5,.5) -- +(1,0);

\begin{footnotesize}
\node at (.25,.6) {4};
\node at (.25,.25) {$-$};
\node at (.75,.6) {3};
\node at (1.25,.6) {2};
\node at (1.25,.25) {$+$};
\node at (2,.75) {1};
\end{footnotesize}

\end{tikzpicture}
\hspace{3em}
\begin{tikzpicture}[scale=0.7]

\draw (0,0) -- +(0,1);
\draw (.5,0) -- +(0,1);
\draw (0,0) -- +(1.5,0);
\draw (0,1) -- +(2.5,0);
\draw (1.5,0) -- +(0,1);
\draw (1,0) -- +(0,1);
\draw (2.5,.5) -- +(0,.5);
\draw (1.5,.5) -- +(1,0);

\begin{footnotesize}
\node at (.25,.6) {4};
\node at (.25,.25) {$-$};
\node at (.75,.6) {3};
\node at (1.25,.6) {2};
\node at (1.25,.25) {$-$};
\node at (2,.75) {1};
\end{footnotesize}

\end{tikzpicture}
\]
\[
\begin{tikzpicture}[scale=0.7]

\draw (0,0) -- +(0,1);
\draw (.5,0) -- +(0,1);
\draw (0,0) -- +(1.5,0);
\draw (0,1) -- +(2.5,0);
\draw (.5,.5) -- +(2,0);
\draw (1.5,0) -- +(0,1);
\draw (2.5,.5) -- +(0,.5);

\begin{footnotesize}
\node at (.25,.6) {4};
\node at (.25,.25) {$+$};
\node at (1,.75) {3};
\node at (1,.25) {2};
\node at (2,.75) {1};
\end{footnotesize}

\end{tikzpicture}
\hspace{3em}
\begin{tikzpicture}[scale=0.7]

\draw (0,0) -- +(0,1);
\draw (.5,0) -- +(0,1);
\draw (0,0) -- +(1.5,0);
\draw (0,1) -- +(2.5,0);
\draw (.5,.5) -- +(2,0);
\draw (1.5,0) -- +(0,1);
\draw (2.5,.5) -- +(0,.5);

\begin{footnotesize}
\node at (.25,.6) {4};
\node at (.25,.25) {$-$};
\node at (1,.75) {3};
\node at (1,.25) {2};
\node at (2,.75) {1};
\end{footnotesize}

\end{tikzpicture}
\hspace{3em}
\begin{tikzpicture}[scale=0.7]

\draw (0,0) -- +(0,1);
\draw (.5,0) -- +(0,1);
\draw (0,0) -- +(1.5,0);
\draw (0,1) -- +(2.5,0);
\draw (.5,.5) -- +(2,0);
\draw (1.5,0) -- +(0,1);
\draw (2.5,.5) -- +(0,.5);

\begin{footnotesize}
\node at (.25,.6) {4};
\node at (.25,.25) {$+$};
\node at (1,.75) {3};
\node at (1,.25) {1};
\node at (2,.75) {2};
\end{footnotesize}

\end{tikzpicture}
\hspace{3em}
\begin{tikzpicture}[scale=0.7]

\draw (0,0) -- +(0,1);
\draw (.5,0) -- +(0,1);
\draw (0,0) -- +(1.5,0);
\draw (0,1) -- +(2.5,0);
\draw (.5,.5) -- +(2,0);
\draw (1.5,0) -- +(0,1);
\draw (2.5,.5) -- +(0,.5);

\begin{footnotesize}
\node at (.25,.6) {4};
\node at (.25,.25) {$-$};
\node at (1,.75) {3};
\node at (1,.25) {1};
\node at (2,.75) {2};
\end{footnotesize}

\end{tikzpicture}
\]
The set $ADT^\text{sgn}((4,2))$ consists of the following four elements:
\[
\begin{tikzpicture}[scale=0.7]

\draw (0,0) -- +(0,1);
\draw (.5,0) -- +(0,1);
\draw (0,0) -- +(1,0);
\draw (0,1) -- +(2,0);
\draw (1,0) -- +(0,1);
\draw (2,.5) -- +(0,.5);
\draw (1,.5) -- +(1,0);

\begin{footnotesize}
\node at (.25,.6) {3};
\node at (.75,.6) {2};
\node at (.75,.25) {$+$};
\node at (1.5,.75) {1};
\end{footnotesize}

\end{tikzpicture}
\hspace{3em}
\begin{tikzpicture}[scale=0.7]

\draw (0,0) -- +(0,1);
\draw (.5,0) -- +(0,1);
\draw (0,0) -- +(1,0);
\draw (0,1) -- +(2,0);
\draw (1,0) -- +(0,1);
\draw (2,.5) -- +(0,.5);
\draw (1,.5) -- +(1,0);

\begin{footnotesize}
\node at (.25,.6) {3};
\node at (.75,.6) {2};
\node at (.75,.25) {$-$};
\node at (1.5,.75) {1};
\end{footnotesize}

\end{tikzpicture}
\hspace{3em}
\begin{tikzpicture}[scale=0.7]

\draw (0,0) -- +(0,1);
\draw (0,0) -- +(1,0);
\draw (0,1) -- +(2,0);
\draw (0,.5) -- +(2,0);
\draw (1,0) -- +(0,1);
\draw (2,.5) -- +(0,.5);

\begin{footnotesize}
\node at (.5,.75) {3};
\node at (.5,.25) {2};
\node at (1.5,.75) {1};
\end{footnotesize}

\end{tikzpicture}
\hspace{3em}
\begin{tikzpicture}[scale=0.7]

\draw (0,0) -- +(0,1);
\draw (0,0) -- +(1,0);
\draw (0,1) -- +(2,0);
\draw (0,.5) -- +(2,0);
\draw (1,0) -- +(0,1);
\draw (2,.5) -- +(0,.5);

\begin{footnotesize}
\node at (.5,.75) {3};
\node at (.5,.25) {1};
\node at (1.5,.75) {2};
\end{footnotesize}

\end{tikzpicture}
\]
\end{ex}

The following Proposition summarizes the discussion above and should be seen as the analog of Proposition~\ref{prop:Spaltenstein-Vargas} for two-row Springer fibers of types $D$ and $C$:

\begin{prop}[{\cite[Lemma 3.2.3 and 3.3.3]{vL89}}] \label{prop:vL_parametrization_of_components}
Let $\lambda=(\lambda_1,\lambda_2)\in\mathcal P_\epsilon(n)$ be a partition and $x\in\mathcal N_{\mathfrak{g}_\epsilon}$ a nilpotent endomorphism of Jordan type $\lambda$.

For every $T\in ADT(\lambda)$ the closure of the preimage of $T$ under $\mathcal S^x_{\beta_\epsilon}$ is a union of disjoint irreducible components of $\mathcal Fl^x_{\beta_\epsilon}$ indexed by all signed domino tableaux which equal $T$ after forgetting the signs. There exists a bijection between the set of all irreducible components of $\mathcal Fl^x_{\beta_\epsilon}$ and $ADT^\mathrm{sgn}(\lambda)$.
\end{prop}

\begin{rem}
For general Jordan types the irreducible components are parameterized by equivalence classes of signed domino tableaux. However, in the case of two-row partitions (which is the case we are primarily interested in) this equivalence relation is trivial. We refer to \cite[\S3.3]{vL89} or \cite[\S3]{Pie04} for details and the more general statement for arbitrary Jordan types.
\end{rem}

\begin{ex}
The two rightmost signed domino tableaux of shape $(5,3)$ in the second row of Example~\ref{ex:signed_domino_tableaux} index the two disjoint irreducible components in Example~\ref{ex:flags_in_irred_comp}. The sign on the vertical domino corresponds to the two choices $\spn(e_1,e_2,f_1,{\bf i}e_3+f_2)$ respectively $\spn(e_1,e_2,f_1,{\bf i}e_3-f_2)$ as the four-dimensional subspace (see also Remark~\ref{rem:companion_flag}).  
\end{ex}

\section{Topological Springer fibers} \label{sec:Topological_Springer_fibers}

In this section we fix a partition $(n-k,k)$ of $n=2m$, $1\leq k\leq m$, labelling a nilpotent orbit of the orthogonal group, i.e.\ either $k=m$ or $k$ is odd. 

\subsection{Definition of topological Springer fibers}

We fix a rectangle in the plane with $m$ vertices evenly spread along the upper horizontal edge of the rectangle. The vertices are labelled by the consecutive integers $1,\ldots,m$ in increasing order from left to right. 

\begin{defi}
A {\it cup diagram} is a non-intersecting diagram inside the rectangle obtained by attaching lower semicircles called {\it cups} and vertical line segments called {\it rays} to the vertices. In doing so we require that every vertex is connected with precisely one endpoint of a cup or ray. Moreover, a ray always connects a vertex with a point on the lower horizontal edge of the rectangle. Additionally, any cup or ray for which there exists a path inside the rectangle connecting this cup or ray to the right edge of the rectangle without intersecting any other part of the diagram is allowed to be decorated with a single dot. 

If the cups and rays of two given cup diagrams are incident with exactly the same vertices (regardless of the precise shape of the cups) and the distribution of dots on corresponding cups and rays coincides in both diagrams we consider them as equal.

We write $\mathbb B^{n-k,k}$ to denote the set of all cup diagrams on $m$ vertices with $\lfloor\frac{k}{2}\rfloor$ cups. This set decomposes as a disjoint union $\mathbb B^{n-k,k} = \mathbb B^{n-k,k}_{\mathrm{even}} \sqcup \mathbb B^{n-k,k}_{\mathrm{odd}}$, where $\mathbb B^{n-k,k}_{\mathrm{even}}$ (resp. $\mathbb B^{n-k,k}_{\mathrm{odd}}$) consists of all cup diagrams with an even (resp.\ odd) number of dots.
\end{defi} 

We usually neither draw the rectangle around the diagrams nor display the vertex labels.

\begin{ex} \label{ex:cup_diagrams}
The set $\mathbb B^{5,3}$ consists of the cup diagrams
\[
\ba=\begin{tikzpicture}[baseline={(0,-.4)}, scale=.8]
\begin{scope}[xshift=3cm]
\draw[thick] (0,0) .. controls +(0,-.5) and +(0,-.5) .. +(.5,0);

\draw[thick] (1,0) -- +(0,-.8);
\draw[thick] (1.5,0) -- +(0,-.8);
\end{scope}
\end{tikzpicture}
\hspace{3.5em}
\bb=\begin{tikzpicture}[baseline={(0,-.4)}, scale=.8]
\begin{scope}[xshift=3cm]
\draw[thick] (.5,0) .. controls +(0,-.5) and +(0,-.5) .. +(.5,0);

\draw[thick] (0,0) -- +(0,-.8);
\draw[thick] (1.5,0) -- +(0,-.8);
\end{scope}
\end{tikzpicture}
\hspace{3.5em}
{\bf c}=\begin{tikzpicture}[baseline={(0,-.4)}, scale=.8]
\begin{scope}[xshift=3cm]
\draw[thick] (1,0) .. controls +(0,-.5) and +(0,-.5) .. +(.5,0);

\draw[thick] (0,0) -- +(0,-.8);
\draw[thick] (.5,0) -- +(0,-.8);
\end{scope}
\end{tikzpicture}
\hspace{3.5em}
{\bf d}=\begin{tikzpicture}[baseline={(0,-.4)}, scale=.8]
\begin{scope}[xshift=3cm]
\draw[thick] (1,0) .. controls +(0,-.5) and +(0,-.5) .. +(.5,0);

\fill[thick] (1.25,-.365) circle(3pt);

\draw[thick] (0,0) -- +(0,-.8);
\draw[thick] (.5,0) -- +(0,-.8);
\fill[thick] (.5,-.4) circle(3pt);
\end{scope}
\end{tikzpicture}
\]
\[
{\bf e}=\begin{tikzpicture}[baseline={(0,-.4)}, scale=.8]
\begin{scope}[xshift=3cm]
\draw[thick] (0,0) .. controls +(0,-.5) and +(0,-.5) .. +(.5,0);

\draw[thick] (1,0) -- +(0,-.8);
\draw[thick] (1.5,0) -- +(0,-.8);
\fill[thick] (1.5,-.4) circle(3pt);
\end{scope}
\end{tikzpicture}
\hspace{3.5em}
{\bf f}=\begin{tikzpicture}[baseline={(0,-.4)}, scale=.8]
\begin{scope}[xshift=3cm]
\draw[thick] (.5,0) .. controls +(0,-.5) and +(0,-.5) .. +(.5,0);

\draw[thick] (0,0) -- +(0,-.8);
\draw[thick] (1.5,0) -- +(0,-.8);
\fill[thick] (1.5,-.4) circle(3pt);
\end{scope}
\end{tikzpicture}
\hspace{3.5em}
{\bf g}=\begin{tikzpicture}[baseline={(0,-.4)}, scale=.8]
\begin{scope}[xshift=3cm]
\draw[thick] (1,0) .. controls +(0,-.5) and +(0,-.5) .. +(.5,0);

\draw[thick] (0,0) -- +(0,-.8);
\draw[thick] (.5,0) -- +(0,-.8);
\fill[thick] (.5,-.4) circle(3pt);
\end{scope}
\end{tikzpicture}
\hspace{3.5em}
{\bf h}=\begin{tikzpicture}[baseline={(0,-.4)}, scale=.8]
\begin{scope}[xshift=3cm]
\draw[thick] (1,0) .. controls +(0,-.5) and +(0,-.5) .. +(.5,0);

\fill[thick] (1.25,-.365) circle(3pt);

\draw[thick] (0,0) -- +(0,-.8);
\draw[thick] (.5,0) -- +(0,-.8);
\end{scope}
\end{tikzpicture}
\]
where the diagrams in the first (resp.\ second) row are the ones in $\mathbb B^{n-k,k}_{\text{even}}$ (resp.\ $\mathbb B^{n-k,k}_{\text{odd}}$).
\end{ex}



Let $\mathbb S^2 \subseteq \mathbb R^3$ be the two-dimensional standard unit sphere on which we fix the points $p=(0,0,1)$ and $q=(1,0,0)$.

Given a cup diagram $\ba\in\mathbb B^{n-k,k}$, define $S_\ba\subseteq\left(\mathbb S^2\right)^m$ as the submanifold consisting of all $(x_1,\ldots,x_m)\in\left(\mathbb S^2\right)^m$ which satisfy the relations $x_i=-x_j$ (resp.\ $x_i=x_j$) if the vertices $i$ and $j$ are connected by an undotted cup (resp.\ dotted cup). Moreover, we impose the relations $x_i=p$ if the vertex $i$ is connected to a dotted ray and $x_i=-p$ (resp.\ $x_i=q$) if $i$ is connected to an undotted ray which is the rightmost ray in $\ba$ (resp.\ not the rightmost ray). Note that $S_\ba$ is homeomorphic to $(\mathbb S^2)^{\lfloor\frac{k}{2}\rfloor}$, i.e.\ each cup of $\ba$ contributes a sphere.

\begin{defi}
The $(n-k,k)$ {\it topological Springer fiber} $\mathcal S^{n-k,k}_D$ of type $D$ is defined as the union 
\[
\mathcal S^{n-k,k}_D := \bigcup_{\ba \in \mathbb B^{n-k,k}}S_\ba \subseteq\left(\mathbb S^2\right)^m.
\]
\end{defi}

\begin{ex}
In the following we discuss in detail the topological Springer fiber $\mathcal S^{5,3}_D$. The submanifolds of $\left(\mathbb S^2\right)^4$ associated with the cup diagrams in $\mathbb B^{5,3}$ (cf.\ Example~\ref{ex:cup_diagrams}) are the following:

\begin{minipage}{.45\textwidth}
\begin{itemize}
\item $S_\ba=\{(x,-x,p,-q)\mid x\in\mathbb S^2\}$
\item $S_{\bf c}=\{(p,-q,x,-x)\mid x\in\mathbb S^2\}$
\end{itemize}
\end{minipage}
\begin{minipage}{.45\textwidth}
\begin{itemize}
\item $S_\bb=\{(p,x,-x,-q)\mid x\in\mathbb S^2\}$
\item $S_{\bf d}=\{(p,q,x,x)\mid x\in\mathbb S^2\}$
\end{itemize}
\end{minipage}

\begin{minipage}{.45\textwidth}
\begin{itemize}
\item $S_{\bf e}=\{(x,-x,p,q)\mid x\in\mathbb S^2\}$
\item $S_{\bf g}=\{(p,q,x,-x)\mid x\in\mathbb S^2\}$
\end{itemize}
\end{minipage}
\begin{minipage}{.45\textwidth}
\begin{itemize}
\item $S_{\bf f}=\{(p,x,-x,q)\mid x\in\mathbb S^2\}$
\item $S_{\bf h}=\{(p,-q,x,x)\mid x\in\mathbb S^2\}$
\end{itemize}
\end{minipage}

\noindent Each of these manifolds is homeomorphic to a two-sphere. Their pairwise intersection is either a point or empty, e.g.\ we have $S_\ba\cap S_\bb=\{(p,-p,p,-q)\}$ and $S_\ba\cap S_{\bf c}=\emptyset$ (cf.\ also Proposition~\ref{prop:top_intersections}). The $(5,3)$ topological Springer fiber of type $D$ is a disjoint union of two Kleinian singularities of type $D_4$, \cite{Slo83}.
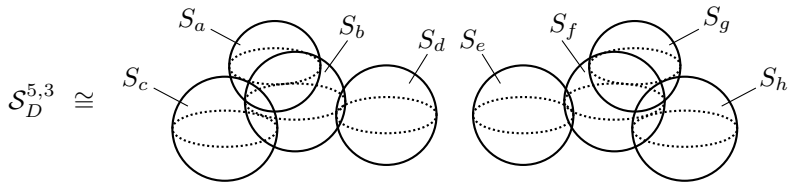
\[
\mathcal S^{5,3}_D\hspace{.4em}\cong\hspace{.4em}
\begin{tikzpicture}[baseline={(0,-.1)}, scale=.6]
\draw[thick] (-.445,.78) circle (1);
\draw[thick] (0,0) circle (1.1);
\draw[thick] (-1.54,-.6) circle (1.15);
\draw[thick] (2,-.3) circle (1.1);
\draw[densely dotted,thick] (0,0) ellipse (1.1 and 0.4);
\draw[densely dotted,thick] (2,-.3) ellipse (1.1 and 0.4);
\draw[densely dotted,thick] (-.445,.78) ellipse (1 and 0.4);
\draw[densely dotted,thick] (-1.54,-.6) ellipse (1.15 and 0.4);

\draw[thick] (7.445,.78) circle (1);
\draw[thick] (7,0) circle (1.1);
\draw[thick] (8.54,-.6) circle (1.15);
\draw[thick] (5,-.3) circle (1.1);
\draw[densely dotted,thick] (7,0) ellipse (1.1 and 0.4);
\draw[densely dotted,thick] (5,-.3) ellipse (1.1 and 0.4);
\draw[densely dotted,thick] (7.445,.78) ellipse (1 and 0.4);
\draw[densely dotted,thick] (8.54,-.6) ellipse (1.15 and 0.4);

\draw (.65,.65) -- (1.15,1.35);
\draw (2.5,.4) -- (2.9,1.05);
\draw (-2.35,-.1) -- (-3.1,.3);
\draw (-1.2,1.25) -- (-1.85,1.6);
\node at (1.25,1.65) {$S_b$};
\node at (3,1.35) {$S_d$};
\node at (-2.25,1.8) {$S_a$};
\node at (-3.5,.5) {$S_c$};

\draw (6.35,.65) -- (5.85,1.35);
\draw (4.5,.4) -- (4.1,1.05);
\draw (9.35,-.1) -- (10.1,.3);
\draw (8.2,1.25) -- (8.85,1.6);
\node at (5.9,1.65) {$S_f$};
\node at (3.9,1.35) {$S_e$};
\node at (9.25,1.8) {$S_g$};
\node at (10.5,.5) {$S_h$};
\end{tikzpicture}
\] 
\end{ex}

\begin{rem} \label{rem:relation_to_ES_model}
A topological model for the Springer fibers of type $D$ corresponding to the partition $(k,k)$ using a slightly different sign convention was introduced in \cite[\S4.1]{ES12}. In order to see that this model is in fact homeomorphic to our topological Springer fiber one considers the involutory diffeomorphism 
\[
I_k\colon (\mathbb S^2)^k \to (\mathbb S^2)^k\,,\,\,(x_1,\ldots,x_k)\mapsto (-x_1,x_2,-x_3,\ldots,(-1)^kx_k).
\] 
Note that if two vertices $i$ and $j$ of a given cup diagram $\ba \in\mathbb B^{k,k}$ are connected by a cup (dotted or undotted) then either $i$ is odd and $j$ is even, or $i$ is even and $j$ is odd. Moreover, the ray in $\ba$ (which exists if and only if $k$ is odd) must be attached to an odd vertex. Thus, the image of the set $S_\ba$ under $I_k$, denoted by $S'_\ba$, consists of all elements $(x_1,\ldots,x_k)\in(\mathbb S^2)^k$ which satisfy the relations $x_i=x_j$ (resp.\ $x_i=-x_j$) if the vertices $i$ and $j$ are connected by an undotted cup (resp.\ dotted cup), $x_i=-p$ if the vertex $i$ is connected to a dotted ray and $x_i=p$ if $i$ is connected to an undotted ray. In particular, we have
\[
I_k\left(\mathcal S^{k,k}_D\right)=\bigcup_{\ba\in\mathbb B^{n-k,k}} I_k\left(S_\ba\right)=\bigcup_{\ba\in\mathbb B^{n-k,k}} S'_\ba
\]
This is precisely the definition of the topological Springer fiber as in \cite[\S4.1]{ES12} (after reversing the order of the coordinates).
\end{rem}

\subsection{Intersections of components}

In the following we provide a combinatorial description of the topology of the pairwise intersections of the submanifolds $S_\ba\subseteq\left(\mathbb S^2\right)^m$, $\ba\in\mathbb B^{n-k,k}$, using circle diagrams (see Proposition~\ref{prop:top_intersections} below). In combination with the homeomorphism $\mathcal S^{n-k,k}_D\cong\mathcal Fl^{n-k,k}_D$ (cf.\ Theorem~\ref{thm:main_result_1}) this yields a combinatorial description of the topology of the pairwise intersections of the irreducible components of $\mathcal Fl^{n-k,k}_D$.   

\begin{defi}
Let $\ba,\bb\in\mathbb B^{n-k,k}$ be cup diagrams. The {\it circle diagram} $\overline{\ba}\bb$ is defined as the diagram obtained by reflecting the diagram $\ba$ in the horizontal middle line of the rectangle and then sticking the resulting diagram, denoted by $\overline{\ba}$, on top of the cup diagram $\bb$, i.e.\ we glue the two diagrams along the horizontal edges of the rectangles containing the vertices (thereby identifying the vertices of $\overline{\ba}$ and $\bb$ pairwise from left to right). In general the diagram $\overline{\ba}\bb$ consists of several connected components each of which is either closed (i.e.\ it has no endpoints) or a line segment. A line segment which contains a ray of $\ba$ and a ray of $\bb$ is called a {\it propagating line}.
\end{defi}

\begin{ex}
Here is an example illustrating the gluing of two cup diagrams in order to obtain a circle diagram:
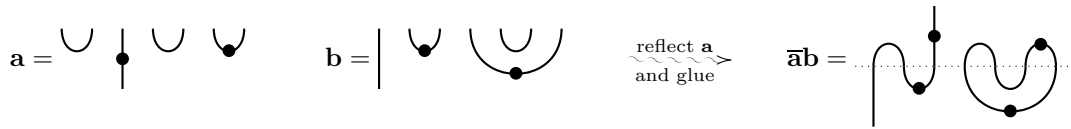
\[
\ba=
\begin{tikzpicture}[scale=.8,baseline={(0,-.5)}]
\draw[thick] (0,0) .. controls +(0,-.5) and +(0,-.5) .. +(.5,0);
\draw[thick] (1.5,0) .. controls +(0,-.5) and +(0,-.5) .. +(.5,0);
\draw[thick] (2.5,0) .. controls +(0,-.5) and +(0,-.5) .. +(.5,0);

\fill[thick] (2.75,-.365) circle(3pt);

\draw[thick] (1,0) -- +(0,-1);
\fill[thick] (1,-.5) circle(3pt);
\end{tikzpicture}
\hspace{3em}
\bb=
\begin{tikzpicture}[scale=.8,baseline={(0,-.5)}]
\draw[thick] (.5,0) .. controls +(0,-.5) and +(0,-.5) .. +(.5,0);
\draw[thick] (1.5,0) .. controls +(0,-1) and +(0,-1) .. +(1.5,0);
\draw[thick] (2,0) .. controls +(0,-.5) and +(0,-.5) .. +(.5,0);

\fill[thick] (2.25,-.74) circle(3pt);
\fill[thick] (.75,-.365) circle(3pt);

\draw[thick] (0,0) -- +(0,-1);
\end{tikzpicture}
\hspace{1.5em}
\begin{xy}
	\xymatrix@=1em{
\ar@{~>}[rrr]^{\textrm{reflect }\ba}_{\textrm{and glue}}  &&&
}
\end{xy}
\hspace{1.5em}
\overline{\ba}\bb=
\begin{tikzpicture}[scale=.8,baseline={(0,0)}]
\draw[dotted] (-.3,0) -- (3.3,0);

\draw[thick] (0,0) .. controls +(0,.5) and +(0,.5) .. +(.5,0);
\draw[thick] (1.5,0) .. controls +(0,.5) and +(0,.5) .. +(.5,0);
\draw[thick] (2.5,0) .. controls +(0,.5) and +(0,.5) .. +(.5,0);
\draw[thick] (.5,0) .. controls +(0,-.5) and +(0,-.5) .. +(.5,0);
\draw[thick] (1.5,0) .. controls +(0,-1) and +(0,-1) .. +(1.5,0);
\draw[thick] (2,0) .. controls +(0,-.5) and +(0,-.5) .. +(.5,0);

\fill[thick] (2.75,.365) circle(3pt);
\fill[thick] (2.25,-.74) circle(3pt);
\fill[thick] (.75,-.365) circle(3pt);

\draw[thick] (1,0) -- +(0,1);
\fill[thick] (1,.5) circle(3pt);
\draw[thick] (0,0) -- +(0,-1);
\end{tikzpicture}
\]
Hence, $\overline{\ba}\bb$ consists of one closed connected component and one line segment which is propagating.
\end{ex}

\begin{prop} \label{prop:top_intersections}
Let $\ba,\bb\in \mathbb B^{n-k,k}$ be cup diagrams. We have $S_\ba\cap S_\bb\neq\emptyset$ if and only if the following conditions hold:
\begin{enumerate}[{(\text{I}}1{)}]
\item Every connected component of $\overline{\ba}\bb$ contains an even number of dots. 
\item Every line segment in $\overline{\ba}\bb$ is propagating.
\end{enumerate}
Furthermore, if $S_\ba\cap S_\bb\neq\emptyset$, then there exists a homeomorphism $S_\ba\cap S_\bb\cong\left(\mathbb S^2\right)^\mathrm{circ}$, where $\mathrm{circ}$ denotes the number of closed connected components of $\overline{\ba}\bb$.

\end{prop}
\begin{proof}
Assume that $S_\ba\cap S_\bb\neq\emptyset$ and let $(x_1,\ldots,x_m)\in S_\ba\cap S_\bb$. Consider a connected component of $\overline{\ba}\bb$ containing the vertices $i_1,\ldots,i_r$ which are ordered such that $i_j$ and $i_{j+1}$ are connected by a cup for all $j\in\{1,\ldots,r-1\}$. If the component is closed (resp.\ a line segment) the total number of cups on the component equals $r$ (resp.\ $r-1$). We denote by $c$ (resp.\ $d$) the total number of undotted (resp.\ dotted) cups of this component.   

\begin{enumerate}[1{)}]
\item If the component is closed we have equalities $x_{i_j}=\pm x_{i_{j+1}}$ for all $j\in\{1,\ldots,r-1\}$ and $x_{i_r}=\pm x_{i_1}$, where the signs depends on whether the cup connecting the respective vertices is dotted or undotted. By successively inserting these equations into each other we obtain $x_{i_1}=(-1)^cx_{i_1}=(-1)^{r-d}x_{i_1}$. In order for this equation to hold, $r-d$ must be even. Since $r$ is even for closed components we deduce that $d$ is even which proves property (I1).
\item If the component is a line segment and $i_1$ is connected to the rightmost ray of either $\ba$ or $\bb$ we have $x_{i_1}\in\{\pm p\}$ and thus $x_{i_r}\in\{\pm p\}$ because $x_{i_1}=(-1)^{(r-1)-d}x_{i_r}$. This shows that $i_r$ is also connected to a rightmost ray. In particular, this line segment is propagating which proves (I2). 

Since $r-1$ is even for a propagating line, $x_{i_1}=(-1)^{(r-1)-d}x_{i_r}$ reduces to $x_{i_1}=(-1)^dx_{i_r}$. If $d$ is even both rays are either dotted or undotted and if $d$ is odd precisely one of the two rays is dotted (otherwise this equation would not be true). In any case, the total number of dots on the line segment is even which shows (I1). 
\item If the component is a line segment and $i_1$ is connected to a ray which is not a rightmost ray in $\ba$ nor $\bb$ we have $x_{i_1}=q$ and thus also $x_{i_r}=q$ because $x_{i_1}=(-1)^{(r-1)-d}x_{i_r}$ as in the previous case. This implies $(-1)^{(r-1)-d}=1$. Moreover, $d=0$ because we already know that the rightmost ray is connected to the rightmost ray. Thus, $r-1$ is even and hence the line is propagating.
\end{enumerate}

In order to prove the implication ``$\Leftarrow$'' we assume that $\ba,\bb\in\mathbb B^{n-k,k}$ are such that (I1) and (I2) are true. Any element $(x_1,\ldots,x_m)\in S_\ba\cap S_\bb$ can be constructed according to the following rules: 
\begin{enumerate}[1{)}] 
\item Given a closed connected component of $\overline{\ba}\bb$ containing the vertices $i_1,\ldots,i_r$ (ordered such that successive vertices are connected by a cup) we fix an element $x_{i_1}\in\mathbb S^2$ and then define $x_{i_2},\ldots,x_{i_r}$ by the relations imposed from $\ba$ and $\bb$. More precisely, under the assumption that $x_{i_1},\ldots,x_{i_j}$ are already constructed we define $x_{i_{j+1}}=\pm x_{i_j}$, $j\in\{1,\ldots,r-1\}$, and proceed inductively.

It remains to check that $x_{i_1}$ and $x_{i_r}$ satisfy the relation imposed by the cup connecting $i_1$ and $i_r$ (by construction $x_{i_1},\ldots,x_{i_r}$ automatically satisfy all the other relations imposed by $\overline{\ba}\bb$). We have $x_{i_1}=(-1)^{(r-1)-d'}x_{i_r}=-(-1)^{d'}x_{i_r}$, where $d'$ denotes the number of dotted cups on the component excluding the cup connecting $i_1$ and $i_r$. If $d'$ is even (resp.\ odd) this cup must be undotted (resp.\ dotted) because of (I1). Hence, the equation gives the correct relation.  
\item For every line segment of $\overline{\ba}\bb$ containing the vertices $i_1,\ldots,i_r$ (again, successive vertices are assumed to be connected by a cup) we define $x_{i_1}=q$ if $i_1$ is not connected to a rightmost ray and $x_{i_1}=-p$ (resp.\ $x_{i_1}=p$) if $i_1$ is connected to an undotted (resp.\ dotted) rightmost ray. We then define coordinates $x_{i_2},\ldots,x_{i_r}$ by the relations coming from the cups in $\ba$ and $\bb$ (as in the case of a closed component). 

We need to check that $x_{i_r}$ satisfies the relation imposed by the ray connected to $i_r$. We have $x_{i_1}=(-1)^{(r-1)-d}x_{i_r}=(-1)^dx_{i_r}$, where $d$ is the number of dotted cups of the line segment ($r-1$ is even because the line segment is propagating by (I2)). 

If $x_{i_1}=-p$ (resp.\ $x_{i_1}=p$), i.e.\ $i_1$ is connected to a rightmost ray which is undotted (resp.\ dotted), then $i_r$ is connected to a rightmost ray as well (otherwise the diagram would not be crossingless) and we have to show that $x_{i_r}=-p$ (resp.\ $x_{i_1}=p$) if the ray connected to $i_r$ is undotted (resp.\ dotted). In case that $d$ is even and $i_1$ is dotted (resp.\ undotted), then $i_r$ is also dotted (resp.\ undotted) and in case that $d$ is even and $i_1$ is dotted (resp.\ undotted), then $i_r$ is undotted (resp.\ dotted). In any case, everything is compatible. 

If $x_{i_1}=q$, i.e.\ $i_1$ is connected to a ray which is not a rightmost ray, then $i_r$ is connected to a ray which is not a rightmost ray either and we have to check that $x_{i_r}=(-1)^dx_{i_1}$ equals $q$. Since there exists a propagating line in $\overline{\ba}\bb$ containing the two rightmost rays of $\ba$ and $\bb$ it follows that $d=0$.
\end{enumerate}
Observe that this construction does not only prove the implication ``$\Leftarrow$'' but it also shows that $S_\ba\cap S_\bb$ is homeomorphic to $\left(\mathbb S^2\right)^\mathrm{circ}$ if $S_\ba\cap S_\bb\neq\emptyset$. 
\end{proof}







\begin{rem} \label{rem:connected_components}
If $\ba \in \mathbb B^{n-k,k}_{\text{even}}$ and $\bb \in \mathbb B^{n-k,k}_{\text{odd}}$ there exists a connected component of $\overline{\ba}\bb$ with an odd number of dots. Thus, it follows directly from Proposition~\ref{prop:top_intersections} that $S_\ba\cap S_\bb=\emptyset$ because $\overline{\ba}\bb$ violates (I1). In particular, the topological Springer fiber $\mathcal S^{n-k,k}_D$ always decomposes as the disjoint union of the spaces $\mathcal S^{n-k,k}_{D,\text{even}} := \bigcup_{\ba \in \mathbb B^{n-k,k}_{\text{even}}}S_\ba$ and $\mathcal S^{n-k,k}_{D,\text{odd}} := \bigcup_{\ba \in \mathbb B^{n-k,k}_{\text{odd}}}S_\ba$.
\end{rem}

\subsection{Some combinatorial bijections} \label{subsec:combinatorial_bijections}

We finish this section by describing some combinatorial bijections which already foreshadow the existence of the homeomorphisms in Theorem~\ref{thm:main_result_1} and Theorem~\ref{thm:main_result_2} on a purely combinatorial level. 

Let $(n-k,k)$ be a two-row partition (not necessarily admissible of type $C$ or $D$) and let $T$ be a standard Young tableau of shape $(n-k,k)$. Let $\psi(T)$ be the unique undecorated cup diagram on $n$ vertices whose left endpoints of cups are precisely the $k$ entries in the lower row of $S$.

\begin{lem}[{\cite[Proposition 3]{SW12}}] 
The assignment $\psi$ defines a bijection
\begin{equation} \label{eq:bij_std_tableaux_undecorated_cup_diags}
\psi\colon SYT(n-k,k) \overset{1:1}{\longleftrightarrow} 
	\begin{Bmatrix}
		{\textrm{undecorated cup diagrams}}\\
		{\textrm{on }n\textrm{ vertices with }k\textrm{ cups}}
	\end{Bmatrix}
\end{equation}
\end{lem}

\begin{ex} 
Via (\ref{eq:bij_std_tableaux_undecorated_cup_diags}) the five standard Young tableaux in Example \ref{ex:std_yng_tab} correspond (in the same order) to the following undecorated cup diagrams on five vertices with two cups:
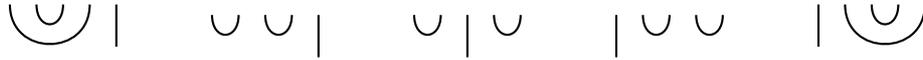
\[
\begin{tikzpicture}[scale=.7]
\begin{scope}[xshift=3cm]
\draw[thick] (0,0) .. controls +(0,-1) and +(0,-1) .. +(1.5,0);
\draw[thick] (.5,0) .. controls +(0,-.5) and +(0,-.5) .. +(.5,0);

\draw[thick] (2,0) -- +(0,-.8);
\end{scope}
\end{tikzpicture}
\hspace{3.5em}
\begin{tikzpicture}[scale=.7]
\begin{scope}[xshift=3cm]
\draw[thick] (0,0) .. controls +(0,-.5) and +(0,-.5) .. +(.5,0);
\draw[thick] (1,0) .. controls +(0,-.5) and +(0,-.5) .. +(.5,0);

\draw[thick] (2,0) -- +(0,-.8);
\end{scope}
\end{tikzpicture}
\hspace{3.5em}
\begin{tikzpicture}[scale=.7]
\begin{scope}[xshift=3cm]
\draw[thick] (0,0) .. controls +(0,-.5) and +(0,-.5) .. +(.5,0);
\draw[thick] (1.5,0) .. controls +(0,-.5) and +(0,-.5) .. +(.5,0);

\draw[thick] (1,0) -- +(0,-.8);
\end{scope}
\end{tikzpicture}
\hspace{3.5em}
\begin{tikzpicture}[scale=.7]
\begin{scope}[xshift=3cm]
\draw[thick] (.5,0) .. controls +(0,-.5) and +(0,-.5) .. +(.5,0);
\draw[thick] (1.5,0) .. controls +(0,-.5) and +(0,-.5) .. +(.5,0);

\draw[thick] (0,0) -- +(0,-.8);
\end{scope}
\end{tikzpicture}
\hspace{3.5em}
\begin{tikzpicture}[scale=.7]
\begin{scope}[xshift=3cm]
\draw[thick] (.5,0) .. controls +(0,-1) and +(0,-1) .. +(1.5,0);
\draw[thick] (1,0) .. controls +(0,-.5) and +(0,-.5) .. +(.5,0);

\draw[thick] (0,0) -- +(0,-.8);
\end{scope}
\end{tikzpicture}
\]
\end{ex}

In the following we will establish a bijection between signed domino tableaux and cup diagrams similar to the one between standard tableaux and undecorated cup diagrams (cf. \cite[Section 5]{ES12}), thereby providing a precise connection between the combinatorics of tableaux and the combinatorics of cup diagrams involved in the definition of the topological Springer fiber.

Note that condition (ADT3) implies that all horizontal dominoes of an admissible domino tableau of shape $(n-k,k)$ of type $D$ have their left box in an even column. In particular, the domino diagram underlying a given signed domino tableau can be constructed by placing the following ``basic building blocks'' side by side:
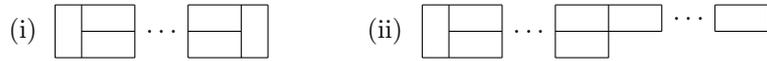
\begin{center}
\begin{tikzpicture}[scale=0.7]
\node at (0,.5) {(i)};

\draw (.6,0) -- +(0,1);
\draw (2.1,0) -- +(0,1);
\draw (.6,0) -- +(1.5,0);
\draw (.6,1) -- +(1.5,0);

\draw (1.1,.5) -- +(1,0);
\draw (1.1,0) -- +(0,1);
\node at (2.6,.5) {$\ldots$};

\draw (3.1,0) -- +(0,1);
\draw (4.6,0) -- +(0,1);
\draw (3.1,0) -- +(1.5,0);
\draw (3.1,1) -- +(1.5,0);

\draw (3.1,.5) -- +(1,0);
\draw (4.1,0) -- +(0,1);

\node at (6.8,.5) {(ii)};

\draw (7.5,0)  -- +(0,1);
\draw (9,0)  -- +(0,1);
\draw (7.5,0)  -- +(1.5,0);
\draw (7.5,1)  -- +(1.5,0);

\draw (8,.5)  -- +(1,0);
\draw (8,0)  -- +(0,1);
\node at (9.5,.5) {$\ldots$};

\draw (10,0)  -- +(0,1);
\draw (10,0)  -- +(1,0);
\draw (10,1)  -- +(2,0);

\draw (10,.5)  -- +(2,0);
\draw (11,0)  -- +(0,1);

\draw (12,.5)  -- +(0,.5);

\node at (12.5,.75) {$\ldots$};

\draw (13,.5)  -- +(0,.5);
\draw (14,.5)  -- +(0,.5);
\draw (13,1)  -- +(1,0);
\draw (13,.5)  -- +(1,0);
\end{tikzpicture}
\end{center}

The left type of building block, called a {\it closed cluster}, consists of a collection of horizontal dominos enclosed by two vertical dominoes. The vertical domino on the left lies in an odd column, whereas the right one lies in an even column. Note that it might happen that a closed cluster has no horizontal dominoes. 

The right building block is called an {\it open cluster}. It consists of a vertical domino lying in an odd column and a bunch of horizontal dominoes to its right. There is at most one open cluster in a given signed domino tableau and if it exists it has to be the rightmost building block.

When decomposing a signed domino tableau into clusters we number them from right to left. 

\begin{ex} \label{ex:clusters_of_signed_domino_tableau}
The signed domino tableau
\[
\begin{tikzpicture}[scale=0.7]

\draw (0,0) -- +(0,1);
\draw (0,0) -- +(8.5,0);
\draw (0,1) -- +(10.5,0);
\draw (10.5,.5) -- +(0,.5);
\draw (8.5,0) -- +(0,1);
\draw (7.5,.5) -- +(3,0);

\draw (0.5,0) -- +(0,1);
\draw (1.5,0) -- +(0,1);
\draw (2.5,0) -- +(0,1);
\draw (3,0) -- +(0,1);
\draw (3.5,0) -- +(0,1);
\draw (4,0) -- +(0,1);
\draw (4.5,0) -- +(0,1);
\draw (5.5,0) -- +(0,1);
\draw (6.5,0) -- +(0,1);
\draw (7,0) -- +(0,1);
\draw (7.5,0) -- +(0,1);
\draw (9.5,.5) -- +(0,.5);

\draw (.5,.5) -- +(2,0);
\draw (4.5,.5) -- +(2,0);

\begin{footnotesize}
\node at (.25,.6) {19};
\node at (.25,.25) {$-$};
\node at (1,.75) {18};
\node at (1,.25) {16};
\node at (2,.75) {17};
\node at (2,.25) {15};
\node at (2.75,.6) {14};

\node at (3.25,.6) {13};
\node at (3.25,.25) {$-$};
\node at (3.75,.6) {12};

\node at (4.25,.6) {11};
\node at (4.25,.25) {$+$};
\node at (5,.75) {10};
\node at (5,.25) {9};
\node at (6,.75) {8};
\node at (6,.25) {7};
\node at (6.75,.6) {6};

\node at (7.25,.6) {5};
\node at (7.25,.25) {$-$};
\node at (8,.75) {4};
\node at (8,.25) {2};
\node at (9,.75) {3};
\node at (10,.75) {1};
\end{footnotesize}

\end{tikzpicture}
\]
is built out of the following four clusters:
\[
\mathcal C_1=\begin{tikzpicture}[baseline={(0,.25)}, scale=0.7]
\draw (0,0) -- +(0,1);
\draw (0,0) -- +(1.5,0);
\draw (0,1) -- +(3.5,0);
\draw (.5,0) -- +(0,1);
\draw (1.5,0) -- +(0,1);
\draw (.5,.5) -- +(3,0);
\draw (2.5,.5) -- +(0,.5);
\draw (3.5,.5) -- +(0,.5);

\begin{footnotesize}
\node at (0.25,.6) {5};
\node at (0.25,.25) {$-$};
\node at (1,.75) {4};
\node at (1,.25) {2};
\node at (2,.75) {3};
\node at (3,.75) {1};
\end{footnotesize}
\end{tikzpicture},\,\,
\mathcal C_2=\begin{tikzpicture}[baseline={(0,.25)}, scale=0.7]
\draw (0,0) -- +(0,1);
\draw (0,0) -- +(3,0);
\draw (0,1) -- +(3,0);
\draw (3,0) -- +(0,1);
\draw (.5,0) -- +(0,1);
\draw (1.5,0) -- +(0,1);
\draw (2.5,0) -- +(0,1);
\draw (.5,.5) -- +(2,0);

\begin{footnotesize}
\node at (0.25,.6) {11};
\node at (0.25,.25) {$+$};
\node at (1,.75) {10};
\node at (1,.25) {9};
\node at (2,.75) {8};
\node at (2,.25) {7};
\node at (2.75,.6) {6};
\end{footnotesize}
\end{tikzpicture},\,\,
\mathcal C_3=\begin{tikzpicture}[baseline={(0,.25)}, scale=0.7]
\draw (0,0) -- +(0,1);
\draw (0,0) -- +(1,0);
\draw (0,1) -- +(1,0);
\draw (.5,0) -- +(0,1);
\draw (1,0) -- +(0,1);

\begin{footnotesize}
\node at (0.25,.6) {13};
\node at (0.25,.25) {$-$};
\node at (.75,.6) {12};
\end{footnotesize}
\end{tikzpicture},\,\,
\mathcal C_4=\begin{tikzpicture}[baseline={(0,.25)}, scale=0.7]
\draw (0,0) -- +(0,1);
\draw (0,0) -- +(3,0);
\draw (0,1) -- +(3,0);
\draw (3,0) -- +(0,1);
\draw (.5,0) -- +(0,1);
\draw (1.5,0) -- +(0,1);
\draw (2.5,0) -- +(0,1);
\draw (.5,.5) -- +(2,0);

\begin{footnotesize}
\node at (0.25,.6) {19};
\node at (0.25,.25) {$-$};
\node at (1,.75) {18};
\node at (1,.25) {16};
\node at (2,.75) {17};
\node at (2,.25) {15};
\node at (2.75,.6) {14};
\end{footnotesize}
\end{tikzpicture}.
\]
\end{ex}

In order to understand the relationship between cup diagrams and signed domino tableaux we first explain how to assign a cup diagram to a cluster of a given tableau $T$. 

Let $\mathcal C$ be a cluster (open or closed) of $T$ and consider its {\it standard tableau part}, i.e. the part of the cluster which remains after removing all vertical dominoes. By viewing the horizontal dominoes as boxes of a Young tableau\footnote{If $\mathcal C$ is a closed cluster one has to subtract the number contained in the right vertical domino from the numbers of the horizontal dominoes in order to obtain a correct filling.} one can use bijection (\ref{eq:bij_std_tableaux_undecorated_cup_diags}) to assign an undecorated cup diagram to the standard tableau part $T_\mathcal C$ of $\mathcal C$. 

The cup diagram associated with the entire cluster is constructed as follows: 
\begin{itemize}
\item If $\mathcal C$ is a closed cluster then we enclose the cup diagram corresponding to the standard tableau part by an additional cup. This cup is dotted if and only if the left vertical domino of the cluster has sign $-$:
\[
\Psi\left(\,\begin{tikzpicture}[baseline={(0,.25)},scale=0.7]
\draw (0,0) -- +(0,1);
\draw (1.5,0) -- +(0,1);
\draw (0,0) -- +(2.5,0);
\draw (0,1) -- +(2.5,0);
\draw (2.5,0) -- +(0,1);

\draw (.5,.5) -- +(2,0);
\draw (.5,0) -- +(0,1);
\node at (3,.5) {$\ldots$};

\draw (3.5,0) -- +(0,1);
\draw (5,0) -- +(0,1);
\draw (3.5,0) -- +(1.5,0);
\draw (3.5,1) -- +(1.5,0);

\draw (3.5,.5) -- +(1,0);
\draw (4.5,0) -- +(0,1);

\begin{footnotesize}
\node at (0.25,.35) {$+$};
\node at (0.25,.15) {$-$};
\end{footnotesize}

\end{tikzpicture}\,\right)
=\begin{cases}\hspace{.5em}
\begin{tikzpicture}[baseline={(0,-.25)}, thick, scale=0.7]

\draw (0,0) .. controls +(0,-1.1) and +(0,-1.1) .. +(1.75,0);
\node at (.875,-.25) {$\psi\left(T_\mathcal C\right)$};

\end{tikzpicture} & \text{ if }+, \\
\hspace{.5em}\begin{tikzpicture}[baseline={(0,-.25)}, thick, scale=0.7]

\draw (0,0) .. controls +(0,-1.1) and +(0,-1.1) .. +(1.75,0);
\node at (.875,-.25) {$\psi\left(T_\mathcal C\right)$};
\fill (.875,-.84) circle(3pt);

\end{tikzpicture} & \text{ if }-.
\end{cases}
\]
\item If $\mathcal C$ is an open cluster we add a ray to the right side of the diagram corresponding to the standard tableau part. This ray is dotted if and only if the vertical domino of the cluster has sign $-$:  
\[
\Psi\left(\,\begin{tikzpicture}[baseline={(0,.25)},scale=0.7]
\draw (0,0)  -- +(0,1);
\draw (1.5,0)  -- +(0,1);
\draw (0,0)  -- +(1.5,0);
\draw (0,1)  -- +(1.5,0);

\draw (.5,.5)  -- +(1,0);
\draw (.5,0)  -- +(0,1);
\node at (2,.5) {$\ldots$};

\draw (2.5,0)  -- +(0,1);
\draw (2.5,0)  -- +(1,0);
\draw (2.5,1)  -- +(2,0);

\draw (2.5,.5)  -- +(2,0);
\draw (3.5,0)  -- +(0,1);

\draw (4.5,.5)  -- +(0,.5);

\node at (5,.75) {$\ldots$};

\draw (5.5,.5)  -- +(0,.5);
\draw (6.5,.5)  -- +(0,.5);
\draw (5.5,1)  -- +(1,0);
\draw (5.5,.5)  -- +(1,0);

\begin{footnotesize}
\node at (0.25,.35) {$+$};
\node at (0.25,.15) {$-$};
\end{footnotesize}

\end{tikzpicture}\,\right)
=\begin{cases}
\begin{tikzpicture}[baseline={(0,-.25)}, thick, scale=0.7]

\node at (0,-.25) {$\psi\left(T_\mathcal C\right)$};
\draw (.875,0) -- +(0,-1);

\end{tikzpicture} & \text{ if }+, \\
\begin{tikzpicture}[baseline={(0,-.25)}, thick, scale=0.7]

\node at (0,-.25) {$\psi\left(T_\mathcal C\right)$};
\draw (.875,0) -- +(0,-1);
\fill (.875,-.74) circle(3pt);

\end{tikzpicture} & \text{ if }-.
\end{cases}
\]
\end{itemize}

Finally, we define $\Psi(T)$ as the cup diagram obtained by placing the cup diagrams $\Psi(\mathcal C_i)$ associated with each cluster side by side, starting with $\Psi(\mathcal C_1)$ as the leftmost piece, followed by $\Psi(\mathcal C_2)$ to its right, etc. This is clearly a well-defined cup diagram. Note that the numbers contained in the horizontal dominos in the lower row and the vertical dominos in an odd column are precisely the left endpoints of the cups in $\Psi(T)$.

\begin{lem}[{\cite[Lemma 5.12]{ES12}}] \label{lem:bijection_tableaux_cups}
Let $(n-k,k)$ be an admissible partition of type $D$. The assignment $\Psi$ explained above defines bijections 
\[
ADT_D^\mathrm{sgn}(n-k,k) \overset{1:1}{\longleftrightarrow}\mathbb B^{n-k,k}
\]
\[
ADT_{D,\mathrm{odd}}^\mathrm{sgn}(n-k,k) \overset{1:1}{\longleftrightarrow}\mathbb B^{n-k,k}_\mathrm{odd}.
\]
\end{lem}

\begin{ex} 
To the signed domino tableau from Example~\ref{ex:clusters_of_signed_domino_tableau}, $\Psi$ assigns the cup diagram
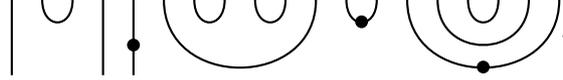
\[
\begin{tikzpicture}[scale=.8,baseline={(0,-.5)}]
\draw[thick] (.5,0) .. controls +(0,-.5) and +(0,-.5) .. +(.5,0);
\draw[thick] (3,0) .. controls +(0,-.5) and +(0,-.5) .. +(.5,0);
\draw[thick] (4,0) .. controls +(0,-.5) and +(0,-.5) .. +(.5,0);
\draw[thick] (5.5,0) .. controls +(0,-.5) and +(0,-.5) .. +(.5,0);
\draw[thick] (7.5,0) .. controls +(0,-.5) and +(0,-.5) .. +(.5,0);

\draw[thick] (7,0) .. controls +(0,-1) and +(0,-1) .. +(1.5,0);

\draw[thick] (2.5,0) .. controls +(0,-1.5) and +(0,-1.5) .. +(2.5,0);
\draw[thick] (6.5,0) .. controls +(0,-1.5) and +(0,-1.5) .. +(2.5,0);

\fill[thick] (2,-.75) circle(3pt);
\fill[thick] (5.75,-.365) circle(3pt);
\fill[thick] (7.75,-1.115) circle(3pt);

\draw[thick] (0,0) -- +(0,-1.25);
\draw[thick] (1.5,0) -- +(0,-1.25);
\draw[thick] (2,0) -- +(0,-1.25);
\end{tikzpicture}.
\]
\end{ex}

\begin{lem} \label{lem:bijection_tableaux_C_D}
Deleting the leftmost vertical domino in a signed domino tableau of shape $(n-k,k)$ gives rise to bijections
\[
ADT_{D,\mathrm{odd}}^\mathrm{sgn}(n-k,k)\overset{1:1}{\longleftrightarrow}ADT_C^{\mathrm{sgn}}(n-k-1,k-1)
\]
\[
ADT_{D,\mathrm{even}}^\mathrm{sgn}(n-k,k)\overset{1:1}{\longleftrightarrow}ADT_C^{\mathrm{sgn}}(n-k-1,k-1). 
\]
\end{lem}

\begin{proof}
This follows easily from the above.
\end{proof}





\section{A smooth variety containing the algebraic Springer fiber} \label{sec:embedded_Springer_fiber}


Let $N>0$ be a large integer (cf.\ Remark \ref{y_remark} for a more accurate description of what is meant by ``large'') and let $z\colon\mathbb C^{2N}\to\mathbb C^{2N}$ be a nilpotent linear endomorphism with two Jordan blocks of equal size, i.e.\ there exists a Jordan basis 
\begin{equation} \label{eq:Jordan_basis_of_z}
\begin{tikzpicture}[baseline={(0,0)}]
\node (e1) at (0,0) {$e_1$};
\node (e2) at (1,0) {$e_2$};
\node (space1) at (2,0) {$\ldots{}^{}$};
\node (eN) at (3.1,0) {$e_N$};
\node (f1) at (5,0) {$f_1$};
\node (f2) at (6,0) {$f_2$};
\node (space2) at (7,0) {$\ldots{}^{}$};
\node (fN) at (8.1,0) {$f_N$.};
\path[->,font=\scriptsize,>=angle 90,bend right]
(e2) edge (e1)
(space1) edge (e2)
(eN) edge (space1)
(f2) edge (f1)
(space2) edge (f2)
(fN) edge (space2);
\end{tikzpicture}
\end{equation}
on which $z$ acts as indicated (the vectors $e_1$ and $f_1$ are sent to zero). We equip $\mathbb C^{2N}$ with a hermitian structure by declaring $e_1,\ldots,e_N,f_1,\ldots,f_N$ to be an orthonormal basis.

Let $e,f$ be the standard basis of $\mathbb C^2$ and let $C \colon \mathbb C^{2N} \to \mathbb C^2$ be the linear map defined by $C(e_i)=e$ and $C(f_i)=f$, $i \in \{1,\ldots,N\}$. Note that $\mathbb C^2$ has the structure of a unitary vector space coming from the standard Hermitian inner product. 

The following lemma is well known \cite[Lemma 2.2]{CK08} (cf.\ also \cite[Lemma 2.1]{Weh09}).

\begin{lem} \label{lem:lin_alg_lem} 
Let $U \subseteq \mathbb C^{2N}$ be a $z$-stable subspace, i.e.\ $zU \subseteq U$, such that $U \subseteq \im(z)$. Then $C$ restricts to a unitary isomorphism $C: z^{-1}U \cap U^{\perp} \xrightarrow\cong \mathbb C^2$.
\end{lem}

Following \cite{CK08} we define a smooth projective variety 
\begin{displaymath}
Y_m:=\big\{\left(F_1,\ldots,F_m\right) \mid F_i \subseteq \mathbb C^{2N} \textrm{has dimension } i ,\, F_1\subseteq\ldots\subseteq F_m,\, zF_i\subseteq F_{i-1}\big\}.
\end{displaymath}

\begin{rem} \label{y_remark}
Note that the conditions $zF_i\subseteq F_{i-1}$ imply
\begin{displaymath}
F_m \subseteq z^{-1}F_{m-1} \subseteq\ldots\subseteq z^{-m}(0) = \spn(e_1,\ldots,e_m,f_1,\ldots,f_m). 
\end{displaymath}
In particular, the variety $Y_m$ is independent of the choice of $N$ as long as $N \geq m$. In particular, we can always assume (by increasing $N$ if necessary) that all the subspaces of a flag in $Y_m$ are contained in the image of $z$. 
\end{rem}


\begin{prop}[{\cite[Theorem 2.1]{CK08}}] \label{homeo}
The map $\phi_m: Y_m \rightarrow (\mathbb P^1)^m$ defined by
\begin{displaymath}
(F_1,\ldots,F_m) \mapsto \left(C(F_1),C(F_2 \cap F_1^\perp),\ldots,C(F_m \cap F_{m-1}^\perp)\right)
\end{displaymath}
is a diffeomorphism. 
\end{prop}

We fix a partition $(n-k,k)$ of $n=2m>0$ labelling a nilpotent orbit of type $D$, $1\leq k\leq m$. Let $E_{n-k,k}\subseteq\mathbb C^{2N}$ be the subspace spanned by $e_1,\ldots,e_{n-k},f_1,\ldots,f_k$. We equip $E_{n-k,k}$ with a bilinear form $\beta^{n-k,k}_D$ defined as follows: 
\begin{itemize}
\item If $k=m$ we define for all $j,j'\in\{1,\ldots,k\}$:
\[
\beta_D^{k,k}(e_{j'},f_j)=\beta_D^{k,k}(f_j,e_{j'})=(-1)^{j-1}\delta_{j+j',k+1},
\]
\[
\beta_D^{k,k}(e_j,e_{j'})=0 \hspace{1.6em}\text{ and }\hspace{1.6em} \beta_D^{k,k}(f_j,f_{j'})=0. 
\]
\item If $k<m$ we define
\[
\beta_D^{n-k,k}(e_i,f_j)=0, \hspace{1.6em} \beta_D^{n-k,k}(e_i,e_{i'})=(-1)^{i-1}\delta_{i+i',n-k+1}, \hspace{1.6em} \beta_D^{n-k,k}(f_j,f_{j'})=(-1)^j\delta_{j+j',k+1},
\]
for all $i,i'\in\{1,\ldots,n-k\}$ and $j,j'\in\{1,\ldots,k\}$. 
\end{itemize}
Note that the bilinear form $\beta^{n-k,k}_D$ is nondegenerate and symmetric. Moreover, a straightforward computation shows that $\beta^{n-k,k}_D(z(v),w)=-\beta^{n-k,k}_D(v,z(w))$ for all $v,w\in E_{n-k,k}$, i.e.\ the restriction $z_{n-k,k}$ of $z$ to the subspace $E_{n-k,k}$ is a nilpotent endomorphism in the orthogonal Lie algebra associated with $\beta^{n-k,k}_D$.

Similarly, we equip $E_{n-k-1,k-1}\subseteq\mathbb C^{2N}$, which is the span of $e_1,\ldots,e_{n-k-1},f_1,\ldots,f_{k-1}$, with a bilinear form $\beta^{n-k-1,k-1}_C$ defined as follows:
\begin{itemize}
\item If $k=m$ we define for all $j,j'\in\{1,\ldots,k-1\}$:
\[
-\beta_C^{k-1,k-1}(e_{j'},f_j)=\beta_C^{k-1,k-1}(f_j,e_{j'})=(-1)^{j-1}\delta_{j+j',k},
\]
\[
\beta_C^{k-1,k-1}(e_j,e_{j'})=0 \hspace{1.6em}\text{ and }\hspace{1.6em} \beta_C^{k-1,k-1}(f_j,f_{j'})=0.
\]
\item If $k<m$ we define $\beta_C^{n-k-1,k-1}(e_i,f_j)=0$ and 
 \[
\beta_C^{n-k-1,k-1}(e_i,e_{i'})=\begin{cases}
(-1)^i\delta_{i+i',n-k} & \text{if }i<i'\\
(-1)^{i-1}\delta_{i+i',n-k} & \text{if }i>i'
\end{cases}
\]
\[
\beta_C^{n-k-1,k-1}(f_j,f_{j'})=\begin{cases}
(-1)^{j-1}\delta_{j+j',k} & \text{if }j<j'\\
(-1)^j\delta_{j+j',k} & \text{if }j>j'
\end{cases}
\]
for all $i,i'\in\{1,\ldots,n-k-1\}$ and $j,j'\in\{1,\ldots,k-1\}$.
\end{itemize}

Note that the bilinear form $\beta^{n-k-1,k-1}_C$ is nondegenerate and symplectic. We also see that the restriction $z_{n-k-1,k-1}$ of $z$ to $E_{n-k-1,k-1}$ is contained in the symplectic Lie algebra associated with $\beta_C^{n-k-1,k-1}$. The following observation is now trivial:

\begin{lem} \label{lem:embedded_Springer_fiber}
We can view the Springer fiber $\mathcal Fl^{n-k,k}_D$ as a subvariety of $Y_m$ via the following identification 
\begin{equation} \label{eq:type_D_embedding}
\mathcal Fl^{n-k,k}_D \cong \left\{(F_1,\dots,F_m) \in Y_m \left|\;\vcenter{\hbox{$F_m$ is contained in $E_{n-k,k}$ and}
																																					 \hbox{isotropic with respect to $\beta^{n-k,k}_D$}}
																														\right.\right\}
\end{equation}
and similarly
\begin{equation} \label{eq:type_C_embedding}
\mathcal Fl^{n-k-1,k-1}_C \cong \left\{(F_1,\dots,F_{m-1}) \in Y_m \left|\;\vcenter{\hbox{$F_{m-1}$ is contained in $E_{n-k-1,k-1}$ and} 
																																										\hbox{isotropic with respect to $\beta^{n-k-1,k-1}_C$}}
																																	  \right.\right\}. 
\end{equation}
\end{lem}

From now on we will always write $\mathcal Fl^{n-k,k}_D$ and $\mathcal Fl^{n-k-1,k-1}_C$ for the embedded Springer varieties via identifications (\ref{eq:type_D_embedding}) and (\ref{eq:type_C_embedding}).


\begin{lem} \label{z_respects_isotropy_lem}
Let $U \subseteq \mathbb C^{2N}$ be a subspace. If $zU$ is contained in $E_{n-k-2,k-2}$ (resp.\ $E_{n-k-3,k-3}$) and isotropic with respect to $\beta^{n-k-2,k-2}_D$ (resp.\ $\beta^{n-k-3,k-3}_C$) then $U$ is contained in $E_{n-k,k}$ (resp.\ $E_{n-k-1,k-1}$) and isotropic with respect to $\beta^{n-k,k}_D$ (resp.\ $\beta^{n-k-1,k-1}_C$).
\end{lem}
\begin{proof}
We only prove the lemma for the type $D$ case since the type $C$ case works completely analogous. By combining the obvious inclusion $E_{n-k-1,k-1} \subseteq E_{n-k,k}$ with the inclusion
\begin{equation} \label{eq:inclusion}
U \subseteq z^{-1}\left(zU\right) \subseteq z^{-1}\left(E_{n-k-2,k-2}\right)=E_{n-k-1,k-1} 
\end{equation}
we obtain $U \subseteq E_{n-k,k}$. Hence, it suffices to show that $U$ is isotropic with respect to $\beta^{n-k,k}_D$. 

Pick two arbitrary elements $v,w \in U \subseteq E_{n-k,k}$ and write
\[
v=\sum_{i=1}^{n-k}\lambda_ie_i+\sum_{j=1}^k\mu_jf_j \hspace{1.6em}\text{ and }\hspace{1.6em} w=\sum_{i=1}^{n-k}\nu_ie_i+\sum_{j=1}^k\xi_jf_j.
\]
Note that $\lambda_{n-k}=\nu_{n-k}=0$ and $\mu_k=\xi_k=0$ because $v,w\in E_{n-k-1,k-1}$ by (\ref{eq:inclusion}). A straightforward calculation (using the definition of the bilinear form) shows that
\begin{equation} \label{eq:first_calculation}
\beta^{n-k,k}_D(v,w)=\begin{cases}
	\sum_{i=2}^{k-1}(-1)^{i+1}\left(\xi_i\lambda_{k-i+1}+\mu_i\nu_{k-i+1}\right) & \text{if }k=m,\\
	\sum_{i=2}^{n-k-1}(-1)^{i+1}\lambda_i\nu_{n-k+1-i}+\sum_{i=2}^{k-1}(-1)^i\mu_i\xi_{k+1-i} & \text{if }k<m.
	\end{cases}
\end{equation}

In order to see that this is zero we apply $z$ to $v,w$ which yields
\[
z(v)=\sum_{i=1}^{n-k-2}\lambda_{i+1}e_i+\sum_{j=1}^{k-2}\mu_{j+1}f_j\,,\,\,z(w)=\sum_{i=1}^{n-k-2}\nu_{i+1}e_i+\sum_{j=1}^{k-2}\xi_{j+1}f_j.
\]
Another computation shows that
\begin{equation} \label{eq:second_calculation}
\beta^{n-k-2,k-2}_D\left(z(v),z(w)\right)=\begin{cases}
	\sum_{i=2}^{k-1}(-1)^i\left(\xi_i\lambda_{k-i+1}+\mu_i\nu_{k-i+1}\right) & \text{if }k=m,\\
	\sum_{i=2}^{n-k-1}(-1)^i\lambda_i\nu_{n-k+1-i}+\sum_{i=2}^{k-1}(-1)^{i-1}\mu_i\xi_{k+1-i} & \text{if }k<m.
	\end{cases}
\end{equation}

By comparing (\ref{eq:first_calculation}) and (\ref{eq:second_calculation}) we deduce that $\beta^{n-k,k}_D(v,w)=-\beta^{n-k-2,k-2}_D(z(v),z(w))$. Since by assumption $zU$ is isotropic with respect to $\beta^{n-k-2,k-2}_D$ we know that the right hand side of this equality must be zero. This proves the lemma.
\end{proof}

\section{Proof of the main theorems} \label{sec:topology_irred_comp}

In this section we prove our main results (see Theorem~\ref{thm:main_result_1} and Theorem~\ref{thm:main_result_2}). We fix an admissible partition $(n-k,k)$ of $n=2m>0$ of type $D$, $1\leq k\leq m$. 

\subsection{The diffeomorphism \texorpdfstring{$\gamma_{n-k,k}$}{gamma}}

In this subsection we define the diffeomorphism $\gamma_{n-k,k}$ and compute the images of the submanifolds $S_\ba\subset\left(\mathbb S^2\right)^m$ for all $\ba\in\mathbb B^{n-k,k}$.

Consider the stereographic projection
\[
\mathbb R^3 \supset\mathbb S^2\setminus\{p\} \xrightarrow\sigma \mathbb C\,,\,\,(x,y,z)\mapsto \frac{x}{1-z}+{\bf i}\frac{y}{1-z}
\]
and the map $\theta\colon\mathbb C\to\mathbb P^1\backslash\{\spn(e)\}$, $\lambda \mapsto \spn(\lambda e+f)$, which can be combined to define a diffeomorphism
\[
\gamma\colon\mathbb S^2 \to \mathbb P^1\,,\,\,(x,y,z) \mapsto \begin{cases}
  \theta\left(\sigma(x,y,z)\right) & \text{if }(x,y,z)\neq p,\\
  \spn(e) & \text{if }(x,y,z)=p.
\end{cases}
\]
This induces a diffeomorphism $\gamma_m \colon \left(\mathbb S^2\right)^m \to \left(\mathbb P^1\right)^m$ on the $m$-fold products by setting 
\[
\gamma_m(x_1,\ldots,x_m):=\left(\gamma(x_1),\ldots,\gamma(x_m)\right).
\]

Moreover, consider the diffeomorphisms $s\colon\mathbb S^2\to\mathbb S^2$, $(x,y,z)\mapsto (x,z,y)$, and $t\colon\mathbb S^2\to\mathbb S^2$, $(x,y,z)\mapsto (z,y,x)$. These yield diffeomorphisms $s_m,t_m\colon\left(\mathbb S^2\right)^m \to\left(\mathbb S^2\right)^m$ by taking $m$-fold products of the respective maps.

We define the diffeomorphisms $\gamma_{n-k,k}\colon\left(\mathbb S^2\right)^m\to\left(\mathbb P^1\right)^m$ as follows:
\[
\gamma_{n-k,k}:=
\begin{cases}
\gamma_m & \text{if }m=k,\\
\gamma_m\circ t_m & \text{if }m-k\text{ is odd},\\
\gamma_m\circ s_m & \text{if }m-k\text{ is even}.
\end{cases}
\]

Given a cup diagram, we write $i \CupConnect j$ (resp. $i \DCupConnect j$) if the vertices $i<j$ are connected by a cup (resp. dotted cup) and $i \RayConnect$ (resp. $i \DRayConnect$) if there is a ray (resp. dotted ray) attached to the vertex $i$. If $\ba\in\mathbb B^{n-k,k}$ and $k\neq m$ let $\rho(\ba)\in\{1,\ldots,m\}$ denote the vertex connected to the rightmost ray in $\ba$.

\begin{defi} \label{defi:topological_P1_components}
Let $\ba \in\mathbb B^{n-k,k}$ be a cup diagram. We define $T_\ba\subset\left(\mathbb P^1\right)^m$ as the set consisting of all $m$-tuples $(l_1,\ldots,l_m) \in (\mathbb P^1)^m$ whose entries satisfy the following list of relations:
\begin{itemize}
\item If $k=m$ we impose the relations
\[
\mathrm{(R1)}\,\,l_i^\perp = l_j  \hspace{.8em} \text{ if } i \CupConnect j \hspace{1.9em} \mathrm{(R3)}\,\,x_i = \spn(f)  \hspace{.8em} \text{ if } i \RayConnect
\]
\[
\mathrm{(R2)}\,\,l_i = l_j  \hspace{.8em} \text{ if } i \DCupConnect j \hspace{1.9em} \mathrm{(R4)}\,\,x_i = \spn(e) \hspace{.8em} \text{ if } i \DRayConnect.
\] 
for all $i,j\in\{1,\ldots,m\}$.
\item If $k\neq m$ we impose the relations 
\[
\textrm{(R1')}\,\,l_i^\perp = l_j  \hspace{.8em} \text{ if } i \CupConnect j \hspace{1.9em} \textrm{(R2')}\,\,l_i = l_j  \hspace{.8em} \text{ if } i \DCupConnect j \hspace{1.9em} \textrm{(R3')}\,\,l_i = \spn(e) \hspace{.8em} \text{ if } i \RayConnect
\]
for all $i,j\in\{1,\ldots,m\}\setminus\{\rho(\ba)\}$ and the additional relation  
\[
\textrm{(R4')}\,\,l_{\rho(\ba)}=\begin{cases}
\spn\left(e+(-1)^{\epsilon}f\right) & \text{if }m-k\text{ is even,}\\
\spn\left({\bf i}e+(-1)^{\epsilon}f\right) & \text{if }m-k\text{ is odd,}
\end{cases}
\]
where $\epsilon=0$ if $\rho(\ba) \DRayConnect$ and $\epsilon=1$ if $\rho(\ba) \RayConnect$.
\end{itemize}
\end{defi}

\begin{lem} \label{lem:sphere_vs_projective_space}
We have an equality of sets $\gamma_{n-k,k}(S_\ba)=T_\ba$ for every $\ba\in\mathbb B^{n-k,k}$.
\end{lem}
\begin{proof}
Since this follows easily from the definitions we omit the proof.
\end{proof}

\subsection{Topology of the irreducible components}

The goal of this subsection is to prove that the diffeomorphism
\[
\left(\mathbb S^2\right)^m\xrightarrow{\gamma_{n-k,k}} \left(\mathbb P^1\right)^m\xrightarrow{\phi_m^{-1}} Y_m
\]
maps each of the submanifolds $S_\ba\subset\left(\mathbb S^2\right)^m$ onto an irreducible component of the Springer fiber $\mathcal Fl^{n-k,k}_D\subset Y_m$. Moreover, if $k>1$, we check that the composition 
\[
\left(\mathbb S^2\right)^m\xrightarrow{\gamma_{n-k,k}} \left(\mathbb P^1\right)^m\xrightarrow{\phi_m^{-1}} Y_m\twoheadrightarrow Y_{m-1}
\]
where $\pi_m\colon Y_m\to Y_{m-1}$, $(F_1,\ldots,F_m)\mapsto(F_1,\ldots,F_{m-1})$, is the morphism of algebraic varieties which forgets the last vector space of a flag, maps the submanifolds $S_\ba$ onto an irreducible component of $\mathcal Fl^{n-k,k}_C\subset Y_{m-1}$. By Lemma~\ref{lem:sphere_vs_projective_space} it suffices to prove the following 

\begin{prop} \label{prop:preimage_contained}
The preimage $\phi_m^{-1}(T_\ba)\subset Y_m$ is an irreducible component of the (embedded) Springer fiber $\mathcal Fl^{n-k,k}_D\subset Y_m$ for all cup diagrams $\ba \in\mathbb B^{n-k,k}$. Moreover, if $k>1$, $\pi_m\left(\phi_m^{-1}\left(T_\ba\right)\right)\subset Y_{m-1}$ is an irreducible component of the (embedded) Springer fiber $\mathcal Fl^{n-k-1,k-1}_C\subset Y_{m-1}$ for all cup diagrams $\ba \in\mathbb B^{n-k,k}_\mathrm{odd}$.
\end{prop}

In order to prove the above proposition (which will occupy most of the remaining section) we proceed by induction on the number of undotted cups. 

\subsubsection{Proof of Proposition~\ref{prop:preimage_contained}: Base case of the induction}

In the following we prove Proposition~\ref{prop:preimage_contained} for all cup diagrams without undotted cups contained in $\mathbb B^{n-k,k}$. It is useful to distinguish two different cases:

\begin{enumerate}
\item If $k$ is odd, $\mathbb B^{n-k,k}$ contains precisely two such diagrams, namely
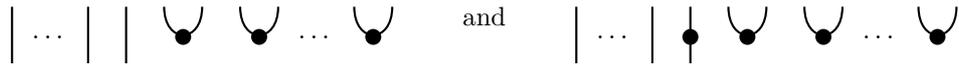
\[
\begin{tikzpicture}[thick,baseline={(0,-.25)},scale=1]
\begin{scope}[xshift=0cm]
\node at (.5,-.4) {\ldots};
\draw (0,0) -- +(0,-.75);
\draw (1,0) -- +(0,-.75);
\draw (1.5,0) -- +(0,-.75);
\draw (2,0) .. controls +(0,-.5) and +(0,-.5) .. +(.5,0);
\fill (2.25,-.4) circle(3pt);
\draw (3,0) .. controls +(0,-.5) and +(0,-.5) .. +(.5,0);
\fill (3.25,-.4) circle(3pt);
\node at (4,-.4) {\ldots};
\draw (4.5,0) .. controls +(0,-.5) and +(0,-.5) .. +(.5,0);
\fill (4.75,-.4) circle(3pt);
\end{scope}
\end{tikzpicture}
\hspace{2.6em}
\text{and}
\hspace{2.6em}
\begin{tikzpicture}[thick,baseline={(0,-.25)},scale=1]
\begin{scope}[xshift=0cm]
\node at (.5,-.4) {\ldots};
\draw (0,0) -- +(0,-.75);
\draw (1,0) -- +(0,-.75);
\draw (1.5,0) -- +(0,-.75);
\fill (1.5,-.4) circle(3pt);
\draw (2,0) .. controls +(0,-.5) and +(0,-.5) .. +(.5,0);
\fill (2.25,-.4) circle(3pt);
\draw (3,0) .. controls +(0,-.5) and +(0,-.5) .. +(.5,0);
\fill (3.25,-.4) circle(3pt);
\node at (4,-.4) {\ldots};
\draw (4.5,0) .. controls +(0,-.5) and +(0,-.5) .. +(.5,0);
\fill (4.75,-.4) circle(3pt);
\end{scope}
\end{tikzpicture}
\]
which consist of $m-k+1$ rays followed by $\frac{k-1}{2}$ dotted cups placed side by side. 
\item If $k$ is even (which implies $k=m$), there is precisely one such cup diagram in $\mathbb B^{k,k}$ (namely the one which consists of $\frac{k}{2}$ successive cups). 
\end{enumerate}
The following lemma treats the extremal case in which the cup diagram consists of rays only. 

\begin{lem} \label{lem:n-1_1_case}
The preimage $\phi_m^{-1}(T_\ba)\subset Y_m$ is an irreducible component of the (embedded) Springer fiber $\mathcal Fl^{n-1,1}_D\subset Y_m$ for all cup diagrams $\ba \in\mathbb B^{n-1,1}$. 
\end{lem}
\begin{proof}
We distinguish between three different cases:\medskip 

Let $m=1$ and let $\ba\in\mathbb B^{1,1}$ be a cup diagram. By Definition \ref{defi:topological_P1_components} we have $T_\ba=\{\spn(f)\}$ (resp.\ $T_\ba=\{\spn(e)\}$) if the ray is undotted (resp.\ dotted). The preimage of $T_\ba$ under the diffeomorphism
\[
\phi_1\colon \begin{Bmatrix}
 {\textrm{one-dimensional subspaces}} \\
 {F_1\subset\mathbb C^{2N}\textrm{ contained in }\ker(z)}
\end{Bmatrix}\xrightarrow\cong \mathbb P^1\,,\,\,F_1\mapsto C(F_1)
\]
is $\phi_1^{-1}(T_\ba)=\{\spn(f_1)\}$ (resp.\ $\phi_1^{-1}(T_\bb)=\{\spn(e_1)\}$) because $\phi_1(\spn(f_1))=C(\spn(f_1))=\spn(f)$ and $C(\spn(e_1))=\spn(e)$. Note that $\spn(f_1),\spn(e_1)\subset E_{1,1}$ are clearly isotropic with respect to $\beta_{1,1}$. Hence, $\phi_1^{-1}(T_\ba)$ is an irreducible component of the (embedded) Springer fiber $\mathcal Fl^{1,1}_D$. \medskip

Assume that $m>1$ is odd and fix a cup diagram $\ba\in\mathbb B^{n-1,1}$. Then $T_\ba$ consists of a single element $(\spn(e),\ldots,\spn(e),\spn(e+(-1)^\epsilon f))\in\left(\mathbb P^1\right)^m$, where $\epsilon=0$ if the rightmost ray is dotted and $\epsilon=1$ if it is undotted. Let $(F_1,\ldots,F_m)\in Y_m$ be the preimage of this element under the diffeomorphism $\phi_m$. Since $F_1\subset\ker(z)$ we have $F_1=\spn(\lambda e_1+\mu f_1)$ for some $\lambda,\mu\in\mathbb C$, $\lambda\neq 0$ or $\mu\neq 0$. The condition $C(F_1)=\spn(e)$ together with 
\[
C(F_1)=C\left(\spn(\lambda e_1+\mu f_1)\right)=\spn(\lambda e+\mu f)
\]
implies $\mu=0$ and hence $F_1=\spn(e_1)$. By induction we assume that we have already shown 
\[
F_1=\spn(e_1),\ldots,F_j=\spn(e_1,\ldots,e_j),\ldots,F_i=\spn(e_1,\ldots,e_i),
\] 
for some $i\in\{1,\ldots,m-1\}$. If $i<m-1$ we are looking for $F_{i+1}$ such that $zF_{i+1}\subset F_i$ and 
\[
\spn(e)=C(F_{i+1}\cap F_i^\perp)=C\left(F_{i+1}\cap\spn(e_{i+1},\ldots,e_{2m-1},f_1)\right).
\] 
By Lemma \ref{lem:lin_alg_lem} this subspace is unique. Since $\spn(e_1,\ldots,e_i,e_{i+1})$ satisfies these properties we deduce $F_{i+1}=\spn(e_1,\ldots,e_i,e_{i+1})$. If $i=m-1$ we replace condition $\spn(e)=C(F_{i+1}\cap F_i^\perp)$ by $\spn(e+(-1)^\epsilon f)=C(F_{i+1}\cap F_i^\perp)$. Note that $\spn(e_1,\ldots,e_{m-1},e_m+(-1)^\epsilon f_1)$ is a possible choice and hence we deduce that this is $F_m$.

In order to check that $F_1,\ldots,F_m$ are isotropic with respect to $\beta^{n-1,1}_D$ we note that the Gram matrix of $\beta^{n-1,1}_D$ restricted to the span of the vectors $e_1,\ldots,e_m,f_1$ is given by
\begin{equation}\label{eq:gram_matrix_ex}
\left(\begin{tabular}{cccc|c}
0&&&0& \\
&&&& \\
&&&0& \\
0&$\dots$&0&1& \\ \hline
&&&&-1
\end{tabular}\right)
\end{equation}
Thus, the vectors $e_1,\ldots,e_{m-1}$ are pairwise orthogonal which shows that $F_1,\ldots,F_{m-1}$ are isotropic. Since $e_m+(-1)^\epsilon f_1$ is clearly orthogonal to $e_1,\ldots,e_{m-1}$ it suffices to compute
\[
\beta^{n-1,1}_D(e_m+(-1)^\epsilon f_1,e_m+(-1)^\epsilon f_1)=\beta^{n-1,1}_D(e_m,e_m)+\beta^{n-1,1}_D(f_1,f_1)=1+(-1)=0
\]
which shows that $F_m$ is isotropic, too.\medskip

If $m$ is even we have $T_\ba=\{(\spn(e),\ldots,\spn(e),\spn({\bf i}e+(-1)^\epsilon f))\}$ and obtain
\[
F_1=\spn(e_1),\ldots,F_{m-1}=\spn(e_1,\ldots,e_{m-1}),F_m=\spn(e_1,\ldots,e_{m-1},{\bf i}e_m+(-1)^\epsilon f_1)
\]
by arguing similarly as in the case in which $m$ is odd. Note that the two non-zero entries in the Gram matrix (\ref{eq:gram_matrix_ex}) are both $-1$ if $m$ is even. Hence, the additional factor of ${\bf i}$ in front of $e_m$ guarantees that $F_m$ is again isotropic.
\end{proof}










\begin{prop}[Special Case of Proposition \ref{prop:preimage_contained}] \label{prop:preimage_no_undotted_cups}
Let $\ba \in \mathbb B^{n-k,k}$ be a cup diagram without undotted cups, $k>1$. 
\begin{enumerate}
\item The preimage $\phi_m^{-1}(T_\ba)\subset Y_m$ is an irreducible component of $\mathcal Fl^{n-k,k}_D\subset Y_m$ contained in the closure of $\left(S^{z_{n-k,k}}_D\right)^{-1}(T)$, where $T\in ADT_D(n-k,k)$ is the $D$-admissible domino tableau obtained by forgetting the signs of $\Psi^{-1}(\ba)$, i.e.\ the domino diagram   
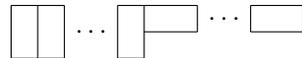
\[
\begin{tikzpicture}[baseline={(0,.25)},scale=0.7]
\draw (0,0)  -- +(0,1);
\draw (.5,0)  -- +(0,1);
\draw (0,0)  -- +(1,0);
\draw (0,1)  -- +(1,0);
\draw (1,0)  -- +(0,1);

\node at (1.5,.5) {$\ldots$};

\draw (2,0)  -- +(0,1);
\draw (2,0)  -- +(.5,0);
\draw (2,1)  -- +(1.5,0);
\draw (2.5,0)  -- +(0,1);
\draw (2.5,.5) -- +(1,0);
\draw (3.5,.5)  -- +(0,.5);

\node at (4,.75) {$\ldots$};

\draw (4.5,.5)  -- +(0,.5);
\draw (5.5,.5)  -- +(0,.5);
\draw (4.5,1)  -- +(1,0);
\draw (4.5,.5)  -- +(1,0);
\end{tikzpicture}
\]
consisting of $k$ vertical dominoes followed by $m-k$ successive horizontal dominoes in the first row together with the unique filling such that (ADT1)-(ADT3) are satisfied.
\item Moreover, $\pi_m(\phi_m^{-1}(T_\ba))\subset Y_{m-1}$ is an irreducible component of $\mathcal Fl^{n-k-1,k-1}_C\subset Y_{m-1}$ contained in the closure of $\left(S^{z_{n-k-k,k-1}}_C\right)^{-1}(T')$, where $T'\in ADT_C(n-k-1,k-1)$ is obtained from $T$ by deleting the leftmost vertical domino.
\end{enumerate}
\end{prop}

\begin{notation}
Let $U\subset E_{n-k-1,k-1}$ be a subspace. In the following we write $U^{\perp_D}$ (resp.\ $U^{\perp_C}$) to denote the orthogonal complement of $U$ with respect to $\beta^{n-k,k}_D$ (resp.\ $\beta^{n-k-1,k-1}_C$). We write $U^\perp$ to denote the orthogonal complement of $U$ in $\mathbb C^{2N}$ with respect to the hermitian structure of $\mathbb C^{2N}$.
\end{notation}

For the proofs of Lemma~\ref{lem:tech_lem_1} and Lemma~\ref{lem:tech_lem_2} it is useful to introduce a technical definition.

\begin{defi} \label{defi:Jordan_system}
Let $(F_1,\ldots,F_m)\in Y_m$ be a flag such that $F_i\subset E_{n-k-1,k-1}$ and $F_i$ is isotropic with respect to both $\beta^{n-k,k}_D$ and $\beta^{n-k-1,k-1}_C$. Moreover, assume that we have Jordan types 
\[
J(z_{n-k,k}^{(i)})=(k-i,k-i) \hspace{2em} J(z_{n-k-1,k-1}^{(i)})=(k-i-1,k-i-1),
\]
where $z_{n-k,k}^{(i)}$ is the endomorphism of $F_i^{\perp_D}/F_i$ induced by $z_{n-k,k}$ and $z_{n-k-1,k-1}^{(i)}$ is the endomorphism of $F_i^{\perp_C}/F_i$ induced by $z_{n-k-1,k-1}$.

A collection of linearly independent vectors in $\mathbb C^{2N}$
\begin{eqnarray} \label{basis_U_induction}
\begin{array}{cc}
&e_1^{(i)},\hspace{.5em} e_2^{(i)}, \hspace{.5em}\ldots\hspace{.5em}, \hspace{.5em}e_{k-i-1}^{(i)}, \hspace{.5em}e_{k-i}^{(i)} \\
&f_1^{(i)},\hspace{.5em} f_2^{(i)}, \hspace{.5em}\ldots\hspace{.5em}, \hspace{.5em}f_{k-i-1}^{(i)}, \hspace{.5em}f_{k-i}^{(i)}
\end{array}
\end{eqnarray}
where $z$ maps each vector to its left neighbor (the leftmost vectors in each row are sent to $F_i$) is called a {\it simultaneous Jordan system for $z_{n-k,k}^{(i)}$ and $z_{n-k-1,k-1}^{(i)}$} if their residue classes (modulo $F_i$) form a Jordan basis of $z_{n-k,k}^{(i)}$ and $z_{n-k-1,k-1}^{(i)}$ (the latter after excluding $e_{k-i}^{(i)}$ and $f_{k-i}^{(i)}$). 

A simultaneous Jordan system as above is called {\it special} if the following additional properties hold:
\begin{enumerate}[{(\text{SJS}}1{)}]

\item \label{item:SJS1} The restriction of $\beta^{n-k,k}_D$ to the simultaneous Jordan system (\ref{basis_U_induction}) is given by the formulae
\[
\beta^{n-k,k}_D\left(f_j^{(i)},e_{j'}^{(i)}\right)=(-1)^{j-1}\delta_{j+j',k-i+1}
\]
for all $j,j'\in\{1,\ldots,k-i\}$. Moreover, the restriction of $\beta^{n-k-1,k-1}_C$ is given by 
\[
-\beta^{n-k-1,k-1}_C\left(e_{j'}^{(i)},f_j^{(i)}\right)=\beta^{n-k-1,k-1}_C\left(f_j^{(i)},e_{j'}^{(i)}\right)=(-1)^{j-1}\delta_{j+j',m-i}
\]
for all $j,j'\in\{1,\ldots,k-i-1\}$.
\item \label{item:SJS2} The vectors in (\ref{basis_U_induction}) form an orthonormal system with respect to the hermitian structure on $\mathbb C^{2N}$ and they are all contained in $F_i^{\perp}$.
\item \label{item:SJS3} We have equalities $C(e_j^{(i)})=C(z(e_j^{(i)}))$ and $C(f_j^{(i)})=C(z(f_j^{(i)}))$ for all $j\in\{2,\ldots,k-i\}$.
\end{enumerate}
\end{defi}


\begin{lem} \label{lem:tech_lem_1}
Assume that $k>1$ is odd and let $\ba\in\mathbb B^{n-k,k}$ be a cup diagram without undotted cups and $(F_1,\ldots,F_m)\in\phi_m^{-1}(T_\ba)$.

Then $F_1,\ldots,F_{m-k+1}$ are contained in $E_{n-k-1,k-1}$ and isotropic with respect to $\beta^{n-k,k}_D$ as well as $\beta^{n-k-1,k-1}_C$, and we have 
\[
J(z_{n-k,k}^{(i)})=(n-k-2i,k), \hspace{2em} J(z_{n-k-1,k-1}^{(i)})=(n-k-2i-1,k-1), 
\]
for all $i\in\{1,\ldots,m-k\}$, and 
\[
J(z_{n-k,k}^{(m-k+1)})=(k-1,k-1), \hspace{2em} J(z_{n-k-1,k-1}^{(m-k+1)})=(k-2,k-2).
\]
Furthermore, there exists a simultaneous Jordan system for $z_{n-k,k}^{(m-k+1)}$ and $z_{n-k-1,k-1}^{(m-k+1)}$ which is special (in the sense of Definition \ref{defi:Jordan_system}).
\end{lem}
\begin{proof}
Let $(F_1,\ldots,F_m) \in\phi_m^{-1}(T_\ba)$, i.e.\ there exists $(l_1,\ldots,l_m)\in T_\ba$ such that 
\begin{equation} \label{eq:equations_to_solve}
C(F_1)=l_1,\,C(F_2\cap F_1^\perp)=l_2,\,\ldots\,,\,C(F_m\cap F_{m-1}^\perp)=l_m. 
\end{equation}
As in the proof of Lemma \ref{lem:n-1_1_case} we distinguish between three different cases:\medskip

If $k=m$ we have $l_1=\spn(f)$ if $\ba$ has an undotted ray and $l_1=\spn(e)$ if $\ba$ has a dotted ray. It follows directly from (\ref{eq:equations_to_solve}) that $F_1=\spn(f_1)$ or $F_1=\spn(e_1)$ both of which are obviously isotropic with respect to $\beta^{k,k}_D$ and $\beta^{k-1,k-1}_C$. 

If $F_1=\spn(e_1)$, we have 
\[
F_1^{\perp_D}=\spn(e_1,\ldots,e_k,f_1,\ldots,f_{k-1}) \hspace{1.6em}\text{and}\hspace{1.6em} F_1^{\perp_C}=\spn(e_1,\ldots,e_{k-1},f_1,\ldots,f_{k-2})
\]
which immediately yields the proposed Jordan types of $z_{k,k}^{(1)}$ and $z_{k-1,k-1}^{(1)}$. Note that the vectors
\[
e_j^{(1)}:=e_{j+1} \text{ and } f_j^{(1)}:=f_j, j \in \{1,\ldots,k-1\}
\]
form a special simultaneous Jordan system for $z_{k,k}^{(1)}$ and $z_{k-1,k-1}^{(1)}$. Similarly, if $F_1=\spn(f_1)$, we have $F_1^{\perp_D}=\spn(e_1,\ldots,e_{k-1},f_1,\ldots,f_k)$ as well as $F_1^{\perp_C}=\spn(e_1,\ldots,e_{k-2},f_1,\ldots,f_{k-1})$ and the vectors 
\[
e_j^{(1)}:=f_{j+1} \text{ and } f_j^{(1)}:=e_j, j \in \{1,\ldots,k-1\} 
\]
form a special simultaneous Jordan basis for $z_{k,k}^{(1)}$ and $z_{k-1,k-1}^{(1)}$.\medskip

If $k\neq m$ and $m-k$ is even we have
\[
l_1=\spn(e),\ldots,l_{m-k}=\spn(e),l_{m-k+1}=\spn(e+(-1)^\epsilon f)
\]
where $\epsilon=0$ if the rightmost ray is dotted and $\epsilon=1$ if it is undotted. By arguing as in Lemma \ref{lem:n-1_1_case} we obtain $F_i=\spn(e_1,\ldots,e_i)$, $1\leq i \leq m-k$, and 
\[
F_{m-k+1} = \spn\left(e_1,\ldots,e_{m-k},e_{m-k+1}+(-1)^\epsilon f_1\right)
\]     
Note that $F_{m-k+1}$ is indeed isotropic with respect to both $\beta^{n-k,k}_D$ and $\beta^{n-k-1,k-1}_C$ (the vectors $e_1,\ldots,e_{m-k+1},f_1$ are isotropic and pairwise orthogonal with respect to either of the two forms).

In order to check the Jordan types first note that $F_i^{\perp_D}=\spn(e_1,\ldots,e_{n-k-i},f_1,\ldots,f_k)$ and $F_i^{\perp_C}=\spn(e_1,\ldots,e_{n-k-i-1},f_1,\ldots,f_{k-1})$ for $i\in\{1,\ldots,m-k\}$. The residue classes of the vectors $e_{i+1},\ldots,e_{n-k-i},f_1,\ldots,f_k$ clearly form a Jordan basis of $z_{n-k,k}^{(i)}$ of the correct type. Similarly, we obtain a Jordan basis of $z_{n-k-1,k-1}^{(i)}$ after deleting $e_{n-k-i}$ and $f_k$. Furthermore, $F_{m-k+1}^{\perp_D}$ is spanned by the linearly independent vectors
\begin{equation} \label{eq:special_vectors_base_induction}
\begin{array} {rl} 
& e_{m-k+1}+(-1)^\epsilon f_1,\,\, \ldots\,\, e_{m-1}+(-1)^\epsilon f_{k-1} \\
e_1,\,\, e_2,\,\, \ldots,\,\, e_{m-k} & \\
& e_{m-k+1}-(-1)^\epsilon f_1,\,\, \ldots\,\, e_{m-1}-(-1)^\epsilon f_{k-1},\,\, e_m-(-1)^\epsilon f_k. 
\end{array}
\end{equation}
Note that $z$ sends a vector to its left neighbor (the leftmost vectors in the first and third row are sent to the rightmost vector in the second row). Furthermore, $F_{m-k+1}^{\perp_C}$ is spanned by the same vectors excluding $e_m-(-1)^\epsilon f_k$ and $e_{m-1}+(-1)^\epsilon f_{k-1}$. It follows directly from (\ref{eq:special_vectors_base_induction}) that $z_{n-k,k}^{(m-k+1)}$ and $z_{n-k-1,k-1}^{(m-k+1)}$ have the correct Jordan type. It is straightforward to check that the vectors 
\[
e_j^{(m-k+1)}:=\frac{1}{\sqrt{2}}\left(e_{m-k+1+j}+(-1)^\epsilon f_{j+1}\right) \text{ and } f_j^{(m-k+1)}:=\frac{1}{\sqrt{2}}\left(e_{m-k+j}-(-1)^\epsilon f_j\right),
\]
where $j \in\{1,\ldots,k-1\}$, form a special simultaneous Jordan system.\medskip

If $m-k$ is odd we similarly get $F_i=\spn(e_1,\ldots,e_i)$, $i\in\{1,\ldots,m-k\}$, and 
\[
F_{m-k+1} = \spn\left(e_1,\ldots,e_{m-k},{\bf i}e_{m-k+1}+(-1)^\epsilon f_1\right)
\]  
which is isotropic with respect to $\beta^{n-k,k}_D$ and $\beta^{n-k-1,k-1}_C$ (again, $e_1,\ldots,e_{m-k+1},f_1$ are isotropic and pairwise orthogonal). As in the case in which $m-k$ is even one can show that the maps $z_{n-k,k}^{(i)}$ have the correct Jordan type for all $i\in\{1,\ldots,m-k+1\}$. We omit the details and claim that 
\[
e_j^{(m-k+1)}:=\frac{1}{\sqrt{2}}\left({\bf i}e_{m-k+1+j}+(-1)^\epsilon f_{j+1}\right) \text{ and } f_j^{(m-k+1)}:=\frac{1}{\sqrt{2}}\left({\bf i}e_{m-k+j}-(-1)^\epsilon f_j\right),
\]
where $j \in\{1,\ldots,k-1\}$, yields a special Jordan system.
\end{proof}


\begin{lem} \label{lem:tech_lem_2}
Let $\ba\in\mathbb B^{n-k,k}$ be a cup diagram without undotted cups and let $(F_1,\ldots,F_m)\in\phi_m^{-1}(T_\ba)$.

If $k>1$ is odd, then the vector spaces $F_{m-k+2},\ldots,F_m$ are isotropic with respect to $\beta^{n-k,k}_D$ and for all $i\in\{m-k+2,\ldots,m\}$ we have
\[
J(z_{n-k,k}^{(i)})=(k-i,k-i).
\]
Moreover, the vector spaces $F_{m-k+2},\ldots,F_{m-1}$ are isotropic with respect to $\beta^{n-k-1,k-1}_C$ and we have
\[
J(z_{n-k-1,k-1}^{(i)})=(k-i-1,k-i-1)
\]
for all $i\in\{m-k+2,\ldots,m-1\}$.

Furthermore, there exists a special simultaneous Jordan system for $z_{n-k,k}^{(i)}$ and $z_{n-k-1,k-1}^{(i)}$ for all $i\in\{m-k+3,m-k+5,\ldots,m-2\}$ (that is all right endpoints of cups except the rightmost one).
\end{lem}
\begin{proof}
Let $(F_1,\ldots,F_m)\in\phi_m^{-1}(T_\ba)$, i.e.\ we have equalities
\[
C(F_1)=l_1,\,C(F_2\cap F_1^\perp)=l_2,\,\ldots\,,\,C(F_m\cap F_{m-1}^\perp)=l_m. 
\]
for some $(l_1,\ldots,l_m)\in T_\ba$. Let $i\in\{m-k+2,\ldots,m-1\}$ be a left endpoint of a dotted cup in $\ba$ and assume by induction that the claims of the lemma are true for $F_{m-k+2},\ldots,F_{i-1}$. The goal is to show the claim for $F_i$ and $F_{i+1}$. 

By induction (or by Lemma~\ref{lem:tech_lem_1} if the cup connecting $i$ and $i+1$ is the leftmost cup) there exists a special simultaneous Jordan system 
\begin{eqnarray}
\label{simultaneous_system}
\begin{array}{cc}
&e_1^{(i-1)},\hspace{.5em} e_2^{(i-1)}, \hspace{.5em}\ldots\hspace{.5em}, \hspace{.5em}e_{k-i}^{(i-1)}, \hspace{.5em}e_{k-i+1}^{(i-1)} \\
&f_1^{(i-1)},\hspace{.5em} f_2^{(i-1)}, \hspace{.5em}\ldots\hspace{.5em}, \hspace{.5em}f_{k-i}^{(i-1)}, \hspace{.5em}f_{k-i+1}^{(i-1)}
\end{array}
\end{eqnarray}
for $z_{n-k,k}^{(i-1)}$ and $z_{n-k-1,k-1}^{(i-1)}$. Since $e_1^{(i-1)}$ and $f_1^{(i-1)}$ form an orthonormal basis of $z^{-1}F_{i-1}\cap F_{i-1}^\perp$it follows from Lemma~\ref{lem:lin_alg_lem} that $C(e_1^{(i-1)})$ and $C(f_1^{(i-1)})$ form an orthonormal basis of $\mathbb C^2$ and we can write $l_i=\spn(C(e_1^{(i-1)})+\mu^{(i)}C(f_1^{(i-1)}))$, where $\mu^{(i)}\in\mathbb C$, or $l_i=\spn(C(f_1^{(i-1)}))$.\smallskip

{\it Check that $F_i$ and $F_{i+1}$ are isotropic:} In order to determine the vector spaces $F_i$ and $F_{i+1}$ we first assume that $l_i=\spn(C(e_1^{(i-1)})+\mu^{(i)}C(f_1^{(i-1)}))$. Consider the $z$-invariant vector space $F_{i-1}\oplus\spn(e_1^{(i-1)}+\mu^{(i)}f_1^{(i-1)})$. Note that 
\begin{align*}
C\left(\left(F_{i-1}\oplus\spn(e_1^{(i-1)}+\mu^{(i)}f_1^{(i-1)})\right)\cap F_{i-1}^\perp\right) &= C\left(\spn(e_1^{(i-1)}+\mu^{(i)}f_1^{(i-1)})\right) \\
																																																		 &= \spn\left(C(e_1^{(i-1)})+\mu^{(i)}C(f_1^{(i-1)})\right)
\end{align*}
Hence, it follows that $F_i=F_{i-1}\oplus\spn(e_1^{(i-1)}+\mu^{(i)}f_1^{(i-1)})$ because $F_i$ is the unique $z$-invariant subspace satisfying $C(F_i\cap F_{i-1}^\perp)=l_i$ (cf.\ Lemma \ref{lem:lin_alg_lem}). Next consider the vector space $F_i\oplus\spn(e_2^{(i-1)}+\mu^{(i)}f_2^{(i-1)})$. Again, this space is $z$-invariant and we have
\begin{align*}
C\left(\left(F_{i-1}\oplus\spn(e_2^{(i-1)}+\mu^{(i)}f_2^{(i-1)})\right)\cap F_{i-1}^\perp\right) &= \spn\left(C(e_2^{(i-1)})+\mu^{(i)}C(f_2^{(i-1)})\right) \\
																																																		 &\overset{(\ref{item:SJS3})}{=} \spn\left(C(z(e_2^{(i-1)}))+\mu^{(i)}C(z(f_2^{(i-1)}))\right) \\
																																																		 &= \spn\left(C(e_1^{(i-1)})+\mu^{(i)}C(f_1^{(i-1)})\right).
\end{align*}
Hence, $F_{i+1}=F_i\oplus\spn(e_2^{(i-1)}+\mu^{(i)}f_2^{(i-1)})$ because $C(F_{i+1}\cap F_i^\perp)=l_i$.

If $l_i=\spn(C(f_1^{(i-1)}))$ we argue as above and obtain $F_i=F_{i-1}\oplus\spn(f_1^{(i-1)})$ as well as $F_{i+1}=F_i\oplus\spn(f_2^{(i-1)})$.

Using formula (\ref{item:SJS1}) and the fact that the vectors in (\ref{simultaneous_system}) are perpendicular to $F_{i-1}$ it is now easy to deduce that $F_i$ and $F_{i+1}$ are isotropic with respect to $\beta^{n-k,k}_D$. Moreover, $F_i$ and $F_{i+1}$ are isotropic with respect to $\beta^{n-k-1,k-1}_C$ if $i<m-1$. If $i=m-1$ only $F_i$ is isotropic with respect to $\beta^{n-k-1,k-1}_C$.\smallskip

{\it Jordan types and special Jordan systems:} It remains to compute the Jordan types of the maps $z_{n-k,k}^{(i)}$, $z_{n-k,k}^{(i+1)}$ and construct a special Jordan system.

Assume that $F_i=F_{i-1}\oplus\spn(e_1^{(i-1)}+\mu^{(i)}f_1^{(i-1)})$ with $\mu^{(i)}\neq 0$ and consider the following linearly independent vectors:
\begin{equation} \label{eq:Jordan_basis}
	\begin{split}
&e_1^{(i-1)}-\mu^{(i)} f_1^{(i-1)}, e_2^{(i-1)}-\mu^{(i)} f_2^{(i-1)}, \ldots , e_{m-(i-1)}^{(i-1)}-\mu^{(i)} f_{m-(i-1)}^{(i-1)}\\
&e_1^{(i-1)}+\mu^{(i)} f_1^{(i-1)}, e_2^{(i-1)}+\mu^{(i)} f_2^{(i-1)}, \ldots , e_{m-(i-1)}^{(i-1)}+\mu^{(i)} f_{m-(i-1)}^{(i-1)} 
	\end{split}
\end{equation}
Again, note that $z$ maps each vector to its left neighbor (the lefmost vectors are sent to $F_{i-1}$). 

Note that $F_i^{\perp_D}$ is the direct sum of $F_{i-1}$ and the span of all the vectors in (\ref{eq:Jordan_basis}) except the rightmost one in the first row. In particular, $z^{(i)}_{n-k,k}$ has the correct Jordan type. Similarly, we see that $F_{i+1}^{\perp_D}$ is the direct sum of $F_{i-1}$ and the span of all vectors in (\ref{eq:Jordan_basis}) except the two rightmost ones in the first row. Hence, $z^{(i+1)}_{n-k,k}$ also has the correct Jordan type.

Moreover, $F_i^{\perp_C}$ is the direct sum of $F_{i-1}$ and the span of all the vectors in (\ref{eq:Jordan_basis}) except the two rightmost ones in the first row and the last one in the second row. In particular, $z^{(i)}_{n-k-1,k-1}$ has the correct Jordan type. Similarly, if $i<m-1$, we see that $F_{i+1}^{\perp_C}$ is the direct sum of $F_{i-1}$ and the span of all vectors in (\ref{eq:Jordan_basis}) except the three rightmost ones in the first and the two rightmost ones in the second row. Hence, $z^{(i+1)}_{n-k-1,k-1}$ also has the correct Jordan type.

In order to finish the induction step we define linearly independent vectors
\[
e_j^{(i+1)}:=\frac{1}{\sqrt{2\mu}}\left(e_j^{(i-1)}-\mu f_j^{(i-1)}\right) \hspace{1em} \text{and} \hspace{1em} f_j^{(i+1)}:=\frac{1}{\sqrt{2\mu}}\left(e_{j+2}^{(i-1)}+\mu f_{j+2}^{(i-1)}\right)
\]
for $j\in\{1,\ldots,m-i-1\}$. It is straightforward to check that these vectors form a special simultaneous Jordan system. 

If $F_i=F_{i-1}\oplus\spn(e_1^{(i-1)})$ or $F_i=F_{i-1}\oplus\spn(f_1^{(i-1)})$ we see that $F_i^{\perp_D}$ is the direct sum of $F_{i-1}$ and the span of all vectors in (\ref{simultaneous_system}) except $f_{m-(i-1)}^{(i-1)}$, if $F_i=F_{i-1}\oplus\spn(e_1^{(i-1)})$, resp.\ $e_{m-(i-1)}^{(i-1)}$, if $F_i=F_{i-1}\oplus\spn(f_1^{(i-1)})$. Similarly, we have that $F_{i+1}^{\perp_D}$ is the direct sum of $F_{i-1}$ and the span of all vectors in (\ref{simultaneous_system}) except $f_{m-i}^{(i-1)}$ and $f_{m-(i-1)}^{(i-1)}$, if $F_i=F_{i-1}\oplus\spn(e_1^{(i-1)})$, resp.\ $e_{m-i}^{(i-1)}$ and $e_{m-(i-1)}^{(i-1)}$, if $F_i=F_{i-1}\oplus\spn(f_1^{(i-1)})$. In both cases we obtain the correct Jordan types for $z_{n-k,k}^{(i)}$ and $z_{n-k,k}^{(i+1)}$. 

If $F_i=F_{i-1}\oplus\spn(e_1^{(i-1)})$ or $F_i=F_{i-1}\oplus\spn(f_1^{(i-1)})$ we see that $F_i^{\perp_C}$ is the direct sum of $F_{i-1}$ and the span of all vectors in (\ref{simultaneous_system}) except the rightmost one in the first (resp.\ second) row and the two rightmost ones in the second (resp.\ first) row if $F_i=F_{i-1}\oplus\spn(e_1^{(i-1)})$ (resp.\ $F_i=F_{i-1}\oplus\spn(f_1^{(i-1)})$). Similarly, we have that $F_{i+1}^{\perp_C}$ is the direct sum of $F_{i-1}$ and the span of all vectors in (\ref{simultaneous_system}) except the three rightmost vectors in the second (resp.\ first) row and the rightmost one in the first (resp.\ second) row if $F_i=F_{i-1}\oplus\spn(e_1^{(i-1)})$ (resp.\ $F_i=F_{i-1}\oplus\spn(f_1^{(i-1)})$). In both cases $z_{n-k-1,k-1}^{(i)}$ has Jordan type $(m-i-1,m-i-1)$ and $z_{n-k-1,k-1}^{(i+1)}$ has Jordan type $(m-i-2,m-i-2)$.

In order to finish we set $e_j^{(i+1)}:=e_j^{(i-1)}$ and $f_j^{(i+1)}:=f_{j+2}^{(i-1)}$ for $j\in\{1,\ldots,m-i-1\}$ in case $F_{i+1}=F_{i-1}\oplus\spn(f_1^{(i-1)},f_2^{(i-1)})$ and we set $e_j^{(i+1)}:=e_{j+2}^{(i-1)}$ and $f_j^{(i+1)}:=f_j^{(i-1)}$ for $j\in\{1,\ldots,m-i-1\}$ in case $F_{i+1}=F_{i-1}\oplus\spn(e_1^{(i-1)},e_2^{(i-1)})$.
\end{proof}

\begin{lem} \label{lem:tech_lem_3}
If $k$ is even, then the vector spaces $F_1,\ldots,F_m$ are isotropic with respect to $\beta^{k,k}_D$ and
\[
J(z_{k,k}^{(i)})=(k-i,k-i)
\]
for all $i\in\{1,\ldots,m\}$. Moreover, the vector spaces $F_1,\ldots,F_{m-1}$ are isotropic with respect to $\beta^{k-1,k-1}_C$ and we have 
\[
J(z_{k-1,k-1}^{(i)})=(k-i-1,k-i-1)
\]
for all $i\in\{1,\ldots,m-1\}$. 
\end{lem}
\begin{proof}
The lemma can be proven by an inductive construction similar to the one in the proof of Lemma \ref{lem:tech_lem_2}. Note that this inductive construction starts with the Jordan basis $e_1,\ldots,e_k,f_1,\ldots,f_k$ of the restriction of $z$ to $E_{k,k}$ as special Jordan system.
\end{proof}

\begin{proof}[Proof (Proposition \ref{prop:preimage_no_undotted_cups}).]
By Lemma \ref{lem:tech_lem_1}, Lemma \ref{lem:tech_lem_2} and Lemma \ref{lem:tech_lem_3} we see that the vector spaces of an arbitrary flag $(F_1,\ldots,F_m)\in\phi_m^{-1}(T_\ba)$ are isotropic with respect to $\beta^{n-k,k}_D$ and that the Jordan types of the sequence of endomorphisms $z_{n-k,k}^{(m-1)},z_{n-k,k}^{(m-2)},\ldots,z_{n-k,k}^{(1)}$ are the following 
\[
(1,1),(2,2),\ldots,(k-1,k-1),(k,k),(k+2,k),\ldots,(n-k,k)
\]
if $k<m$, and 
\[
(1,1),(2,2),\ldots,(k-1,k-1),(k,k)
\]
if $m=k$. Moreover, the vector spaces of the flag $(F_1,\ldots,F_{m-1})=\pi_m(F_1,\ldots,F_m)\in\pi_m\left(\phi_m^{-1}(T_\ba)\right)$ are isotropic with respect to $\beta^{n-k-1,k-1}_C$ and the maps $z_{n-k-1,k-1}^{(m-2)},z_{n-k-1,k-1}^{(m-2)},\ldots,z_{n-k-1,k-1}^{(1)}$ have Jordan types
\[
(1,1),(2,2),\ldots,(k-2,k-2),(k-1,k-1),(k+1,k-1),\ldots,(n-k-1,k-1)
\]
if $k<m$, and 
\[
(1,1),(2,2),\ldots,(k-2,k-2),(k-1,k-1)
\]
if $m=k$. Thus, it follows that $S^{z_{n-k,k}}_D(F_1,\ldots,F_m)=T$ as well as $S^{z_{n-k-1,k-1}}_C(F_1,\ldots,F_{m-1})=T'$ and $(F_1,\ldots,F_m)$ is contained in the closure of $\left(S^{z_{n-k,k}}_D\right)^{-1}(T)$ and $(F_1,\ldots,F_{m-1})$ is contained in the closure of $\left(S^{z_{n-k-1,k-1}}_C\right)^{-1}(T')$. In particular, since $(F_1,\ldots,F_m)$ is an arbitrarily chosen flag in $\phi_m^{-1}(T_\ba)$, we have proven the inclusion of $\phi_m^{-1}(T_\ba)$ in the closure of $\left(S^{z_{n-k,k}}_D\right)^{-1}(T)$ and $\pi_m\left(\phi_m^{-1}(T_\ba)\right)$ in the closure of $\left(S^{z_{n-k-1,k-1}}_C\right)^{-1}(T')$.\medskip

Finally, note that $\phi_m^{-1}(T_\ba)$ is a closed subvariety of $\mathcal Fl^{n-k,k}_D$ which is connected because it is the preimage of $T_\ba$ (which is obviously connected) under a diffeomorphism. Thus, $\phi_m^{-1}(T_\ba)$ must be contained in precisely one of the irreducible components whose disjoint union is the closure of $(S^{z_{n-k,k}}_D)^{-1}(T)$. Since the dimension of $\phi_m^{-1}(T_\ba)$ equals the dimension of $\mathcal Fl^{n-k,k}$ it follows that $\phi_m^{-1}(T_\ba)$ equals the irreducible component of $\mathcal Fl^{n-k,k}_D$ in which it is contained.

It follows that $\pi_m\left(\phi_m^{-1}(T_\ba)\right)$ is also an irreducible closed subvariety of $\mathcal Fl^{n-k-1,k-1}_C$ of the correct dimension.    
\end{proof}

\subsubsection{Proof of Proposition~\ref{prop:preimage_contained}: Inductive step}

In \cite[Section 2]{CK08} the authors introduce a smooth subvariety $X^i_m \subseteq Y_m$, $i \in \{1,\ldots,m-1\}$, defined by 
\[
X^i_m:=\{(F_1,\ldots,F_m) \in Y_m \mid F_{i+1}=z^{-1}F_{i-1}\},
\]
and a surjective morphism of varieties $q_m^i \colon X^i_m \twoheadrightarrow Y_{m-2}$ given by
\[
(F_1,\ldots,F_m) \mapsto \left(F_1,\ldots,F_{i-1},zF_{i+2},\ldots,zF_m\right).
\]
We want to make use of the following lemma (cf.\ \cite[Theorem 2.1]{CK08} and \cite[Lemma 2.4]{Weh09}):

\begin{lem}
The diffeomorphism $\phi_m$ maps $X^i_m$ bijectively to the set
\begin{equation} \label{eq:defi_A}
A_m^i:=\{(l_1,\ldots,l_m)\in(\mathbb P^1)^m\mid l_{i+1}=l_i^\perp\}
\end{equation}
and we have a commutative diagram
\begin{equation} \label{eq:comm_diag}
\begin{tikzcd} 
X^i_m \arrow[two heads]{rr}{q_m^i} \arrow{d}{\cong}[left]{\phi_m\vert_{X_m^i}} &&Y_{m-2} \arrow{d}{\phi_{m-2}}[left]{\cong}\\
A_m^i \arrow[two heads]{rr}{f^i_m\vert_{A_m^i}} &&\left(\mathbb P^1\right)^{m-2} 
\end{tikzcd}
\end{equation}
where $f^i_m\colon\left(\mathbb P^1\right)^m \twoheadrightarrow\left(\mathbb P^1\right)^{m-2}$ is the map which forgets the coordinates $i$ and $i+1$. The orthogonal complement in (\ref{eq:defi_A}) is taken with respect to the hermitian structure of $\mathbb C^2$.
\end{lem}

We prove the proposition by induction on the number of undotted cups in $\ba$. If there is no undotted cup in $\ba$ the claim follows from Lemma~\ref{lem:n-1_1_case} and Proposition~\ref{prop:preimage_no_undotted_cups}. Hence, we may assume that there exists an undotted cup in $\ba$. Then there exists a cup connecting neighboring vertices $i$ and $i+1$. Let $\ba'\in\mathbb B^{n-k-2,k-2}$ be the cup diagram obtained by removing this cup. 

We have
\begin{equation} \label{eq:preimage_under_q}
\left(q_m^i\right)^{-1}\left(\phi_{m-2}^{-1}(T_{\ba'})\right)=\phi_m^{-1}\left(\left(f_m^i\right)^{-1}\left(T_{\ba'}\right)\right)=\phi_m^{-1}(T_\ba),
\end{equation} 
where the first equality followes directly from the commutativity of the diagram (\ref{eq:comm_diag}) and the second one is the obvious fact that $\left(f_m^i\right)^{-1}\left(T_{\ba'}\right)=T_\ba$. Thus, $\phi_m^{-1}(T_\ba)\subseteq Y_m$ is a closed subvariety because it is the preimage of the closed subvariety $\phi_{m-2}^{-1}(T_{\ba'})$ (which is even an irreducible component of $\mathcal Fl^{n-k-2,k-2}_D$ by induction) under the morphism $q_m^i$.

Thus, by (\ref{eq:preimage_under_q}), for any given flag $(F_1,\ldots,F_m)\in\phi_m^{-1}(T_\ba)$, we have  
\[
q_{m,i}(F_1,\ldots,F_m)=\left(F_1,\ldots,F_{i-1},zF_{i+2},\ldots,zF_m\right) \in \phi_{m-2}^{-1}(T_{\ba'})
\]
which shows that $zF_m$ is containd in $E_{n-k-2,k-2}$ and isotropic with respect to $\beta^{n-k-2,k-2}_D$ because $\phi_{m-2}(T_{\ba'}) \subset \mathcal Fl^{n-k-2,k-2}_D$ by induction. Hence, by Lemma \ref{z_respects_isotropy_lem}, $F_m$ is contained in $E_{n-k,k}$ and isotropic with respect to $\beta^{n-k,k}_D$. Thus, we have proven that $\phi_m^{-1}(T_\ba)\subseteq Y_m$ is an algebraic subvariety contained in $\mathcal Fl^{n-k,k}_D$. 

Note that $\phi_m^{-1}(T_\ba)$ is smooth and connected because it is the preimage of $T_\ba$ (which is obviously smooth and connected) under a diffeomorphism. This shows that $\phi_m^{-1}(T_\ba)$ is irreducible. Finally, the dimension of the variety $\phi_m^{-1}(T_\ba)$ obviously equals the dimension of $\mathcal Fl^{n-k,k}_D$ (because the manifold $T_\ba$ has the correct dimension). To sum up, $\phi_m^{-1}(T_\ba)$ must therefore be an irreducible component of the (embedded) Springer variety $\mathcal Fl^{n-k,k}_D$.

Let $\ba\in\mathbb B^{n-k,k}_\mathrm{odd}$. If $\ba$ only has one undotted cup which connects the vertices $m-1$ and $m$ it is easy to see that $\pi_m\left(\phi_m^{-1}(T_\ba)\right)$ equals $\pi_m\left(\phi_m^{-1}(T_\bb)\right)$, where $\bb$ denotes the cup diagram obtained by exchanging the undotted cup connecting $m-1$ and $m$ with a dotted cup. Then the claim follows from Proposition~\ref{prop:preimage_no_undotted_cups}. Analogously, if $i\neq m-1$, by applying $\pi_m$ to (\ref{eq:preimage_under_q}), it is easy to see that we also have an equality of sets 
\[
\pi_m\left(\phi_m^{-1}(T_\ba)\right)=\pi_m\left((q_m^i)^{-1}\left(\phi_{m-2}^{-1}(T_{\ba'})\right)\right)=(q_{m-1}^i)^{-1}\left(\pi_{m-2}\left(\phi_{m-2}^{-1}(T_{\ba'})\right)\right).
\]
Now the claim follows as for the type $D$ Springer fiber. This finishes the proof of Proposition~\ref{prop:preimage_contained}.\hfill $\square$

\subsection{Gluing the irreducible components} \label{sec:proof_main_thms}

Now we state and prove our main results. 

\begin{thm} \label{thm:main_result_1}
The diffeomorphism $\left(\mathbb S^2\right)^m\xrightarrow{\gamma_{n-k,k}}\left(\mathbb P^1\right)^m\xrightarrow{\phi_m^{-1}} Y_m$ restricts to a homeomorphism
\[
\mathcal S^{n-k,k}_D\xrightarrow\cong\mathcal Fl^{n-k,k}_D
\]
such that the images of the $S_\ba$ under this homeomorphism are precisely the irreducible components of $\mathcal Fl^{n-k,k}_D$ for all $\ba\in\mathbb B^{n-k,k}$.
\end{thm}
\begin{proof}
We know that the image of $\mathcal S^{n-k,k}_D\subset\left(\mathbb S^2\right)^m$ under the diffeomorphism $\phi_m^{-1}\circ\gamma_{n-k,k}$ is given by 
\[
\phi_m^{-1}\left(\gamma_{n-k,k}\left(\mathcal S^{n-k,k}_D\right)\right)=\bigcup_{\ba \in\mathbb B^{n-k,k}}\phi_m^{-1}\left(\gamma_{n-k,k}\left(S_\ba\right)\right)\overset{\text{Lemma}}{\underset{\ref{lem:sphere_vs_projective_space}}{=}}\bigcup_{\ba \in\mathbb B^{n-k,k}}\phi_m^{-1}\left(T_\ba\right).
\]
If $\ba\neq\bb$ we obviously have $T_\ba\neq T_\bb$ and thus also $\phi_m^{-1}(T_\ba)\neq \phi_m^{-1}(T_\bb)$ because $\phi_m$ is bijective. In combination with Proposition~\ref{prop:preimage_contained} this implies that $\bigcup_{\ba \in\mathbb B^{n-k,k}} \phi_m^{-1}(T_\ba)$ is a union of irreducible components of $\mathcal Fl^{n-k,k}_D$ which are pairwise different. Since the cup diagrams in $\mathbb B^{n-k,k}$ are in bijective correspondence with the irreducible components of the Springer fiber $\mathcal Fl^{n-k,k}_D$ (cf.\ Proposition~\ref{prop:vL_parametrization_of_components} and Lemma~\ref{lem:bijection_tableaux_cups}) we deduce that $\bigcup_{\ba \in\mathbb B^{n-k,k}} \phi_m^{-1}(T_\ba)$ is the entire (embedded) Springer fiber. In particular, the restriction of the diffeomorphism $\phi_m^{-1}\circ\gamma_{n-k,k}$ to $\mathcal S^{n-k,k}_D$ yields the desired homeomorphism as claimed in the theorem.
\end{proof}


We define $\mathcal Fl^{n-k,k}_{D,\mathrm{odd}}$ as the image of $\mathcal S^{n-k,k}_{D,\mathrm{odd}}$ and $\mathcal Fl^{n-k,k}_{D,\mathrm{odd}}$ as the image of $\mathcal S^{n-k,k}_{D,\mathrm{even}}$ under the homeomorphism of Theorem~\ref{thm:main_result_1}. It follows from Remark~\ref{rem:connected_components} that these are precisely the two connected components of $\mathcal Fl^{n-k,k}_D$. Since they are isomorphic we restrict ourselves to $\mathcal Fl^{n-k,k}_{D,\mathrm{odd}}$ (the results are also true for $\mathcal Fl^{n-k,k}_{D,\mathrm{even}}$). 

\begin{thm} \label{thm:main_result_2}
The morphism of algebraic varieties $Y_m\xrightarrow{\pi_m}Y_{m-1}$ restricts to a homeomorphism (even an isomorphism of algebraic varieties)
\[
\mathcal Fl^{n-k,k}_{D,\mathrm{odd}}\xrightarrow\cong\mathcal Fl^{n-k-1,k-1}_C,
\]
i.e.\ $\mathcal Fl^{n-k-1,k-1}_C$ is isomorphic to one of the two (isomorphic) connected components of $\mathcal Fl^{n-k,k}_D$. In particular, in combination with Theorem~\ref{thm:main_result_1}, we obtain a homeomorphism
\[
\mathcal S^{n-k,k}_{D,\mathrm{odd}}\cong\mathcal Fl^{n-k-1,k-1}_C
\]
and thus an explicit topological model for $\mathcal Fl^{n-k-1,k-1}_C$.
\end{thm}
\begin{proof}
Since $\mathcal Fl^{n-k,k}_{D,\mathrm{odd}}$ is a connected components of $\mathcal Fl^{n-k,k}_D$ it follows directly from Remark~\ref{rem:companion_flag} that the restriction of $\pi_m$ to $\mathcal Fl^{n-k,k}_{D,\mathrm{odd}}\subset Y_m$ defines a continuous injection with image
\[
\pi_m\left(\mathcal Fl^{n-k,k}_{D,\mathrm{odd}}\right)=\bigcup_{\ba \in\mathbb B^{n-k,k}_\mathrm{odd}}\pi_m\left(\phi_m^{-1}\left(T_\ba\right)\right).
\]
By Proposition~\ref{prop:preimage_contained} this is a union of irreducible components of $\mathcal Fl^{n-k-1,k-1}_C$ which are pairwise different (because the restriction of $\pi_m$ to $\mathcal Fl^{n-k,k,}_{D,\mathrm{odd}}$ is bijective and the $\phi_m^{-1}\left(T_\ba\right)$ are pairwise different irreducible components of $\mathcal Fl^{n-k,k}_D$ by Theorem~\ref{thm:main_result_1}). Recall that the irreducible components of $\mathcal Fl^{n-k-1,k-1}_C$ are in bijective correspondence with the cup diagrams in $\mathbb B^{n-k,k}_\mathrm{odd}$ (combine Proposition~\ref{prop:vL_parametrization_of_components} with the bijections Lemma~\ref{lem:bijection_tableaux_cups} and Lemma~\ref{lem:bijection_tableaux_C_D}).   
Hence, the image of $\mathcal Fl^{n-k,k}_{D,\mathrm{odd}}$ under $\pi_m$ equals $\mathcal Fl^{n-k-1,k-1}_C\subset Y_{m-1}$. In particular, $\pi_m$ restricts to the desired homeomorphism.

Note that the homeomorphism $\pi_m\colon\mathcal Fl^{n-k,k}_{D,\mathrm{odd}}\xrightarrow\cong\mathcal Fl^{n-k-1,k-1}_C$, which is even an isomorphism of algebraic varieties.
\end{proof}

\bibliography{litlist}

\providecommand{\bysame}{\leavevmode\hbox to3em{\hrulefill}\thinspace}
\providecommand{\MR}{\relax\ifhmode\unskip\space\fi MR }
\providecommand{\MRhref}[2]{%
  \href{http://www.ams.org/mathscinet-getitem?mr=#1}{#2}
}
\providecommand{\href}[2]{#2}
\begin{thebibliography}{Weh09}

\bibitem[CG97]{CG97}
N.~Chriss and V.~Ginzburg, \emph{Representation theory and complex geometry},
  Birkh{\"a}user, 1997.

\bibitem[CK08]{CK08}
S.~Cautis and J.~Kamnitzer, \emph{Knot homology via derived categories of
  coherent sheaves. {I}. {T}he $\mathfrak{sl}$(2)-case}, Duke Math. J.
  \textbf{142} (2008), 511--588.

\bibitem[ES12]{ES12}
M.~Ehrig and C.~Stroppel, \emph{2-row {S}pringer fibres and {K}hovanov diagram
  algebras for type {D}}, 2012, to appear in Canadian J. Math.

\bibitem[ES13]{ES13}
\bysame, \emph{Diagrammatic description for the categories of perverse sheaves
  on isotropic {G}rassmannians}, ar{X}iv:1306.4043, 2013, to appear in Selecta
  Math. (N.S.).

\bibitem[Fun03]{Fun03}
F.~Fung, \emph{On the topology of components of some {S}pringer fibers and
  their relation to {K}azhdan-{L}usztig theory}, Adv. Math. \textbf{178}
  (2003), no.~2, 244--276.

\bibitem[Ger61]{Ger61}
M.~Gerstenhaber, \emph{Dominance over the classical groups}, Ann. of Math.
  \textbf{74} (1961), no.~3, 532--569.

\bibitem[HL14]{HL14}
A.~Henderson and A.~Licata, \emph{Diagram automorphisms of quiver varieties},
  Adv. Math. \textbf{267} (2014), 225--276.

\bibitem[Kho00]{Kho00}
M.~Khovanov, \emph{A categorification of the {J}ones polynomial}, Duke Math. J.
  \textbf{101} (2000), no.~3, 359--426.

\bibitem[Kho02]{Kho02}
\bysame, \emph{A functor-valued invariant of tangles}, Alg. Geom. Top.
  \textbf{2} (2002), 665--741.

\bibitem[Kho04]{Kho04}
\bysame, \emph{Crossingless matchings and the (n,n) {S}pringer varieties},
  Communications in Contemporary Math. \textbf{6} (2004), 561--577.

\bibitem[LS13]{LS13}
T.~Lejczyk and C.~Stroppel, \emph{A graphical description of {$(D_n,A_{n-1})$}
  {K}azhdan-{L}usztig polynomials}, Glasgow Math. J. \textbf{55} (2013), no.~2,
  313--340.

\bibitem[Pie04]{Pie04}
T.~Pietraho, \emph{Components of the {S}pringer fiber and domino tableaux}, J.
  Algebra \textbf{272} (2004), no.~2, 711--729.

\bibitem[RT11]{RT11}
H.M. Russell and J.~Tymoczko, \emph{Springer representations on the {K}hovanov
  {S}pringer varieties}, Math. Proc. Cambridge Philos. Soc. \textbf{151}
  (2011), 59--81.

\bibitem[Rus11]{Rus11}
H.M. Russell, \emph{A topological construction for all two-row {S}pringer
  varieties}, Pacific J. Math. \textbf{253} (2011), 221--255.

\bibitem[Slo83]{Slo83}
P.~Slodowy, \emph{Platonic solids, {K}leinian singularities, and {L}ie groups},
  Lecture Notes in Mathematics, vol. 1008, pp.~102--138, Springer-Verlag,
  Berlin, 1983.

\bibitem[Spa76]{Spa76}
N.~Spaltenstein, \emph{The fixed point set of a unipotent transformation on the
  flag manifold}, Nederl. Akad. Wetensch. Proc. Ser. A, vol.~79, 1976,
  pp.~452--456.

\bibitem[Spa82]{Spa82}
\bysame, \emph{Classes {U}nipotentes et {S}ous-groupes de {B}orel}, Lecture
  Notes in Mathematics, vol. 946, Springer-Verlag, 1982.

\bibitem[SW12]{SW12}
C.~Stroppel and B.~Webster, \emph{2-block {S}pringer fibers: convolution
  algebras and coherent sheaves}, Comment. Math. Helv. \textbf{87} (2012),
  477--520.

\bibitem[Var79]{Var79}
J.A. Vargas, \emph{Fixed points under the action of unipotent elements of
  {SL}$_n$ in the flag variety}, Bol. Soc. Mat. Mexicana \textbf{24} (1979),
  no.~1, 1--14.

\bibitem[vL89]{vL89}
M.~van Leeuwen, \emph{A {R}obinson-{S}chensted algorithm in the geometry of
  flags for classical groups}, Ph{D} thesis, Rijksuniversiteit Utrecht, 1989.

\bibitem[Weh09]{Weh09}
S.~Wehrli, \emph{A remark on the topology of (n,n) {S}pringer varieties},
  ar{X}iv:0908.2185, 2009.

\bibitem[Wil37]{Wil37}
J.~Williamson, \emph{The {C}onjunctive {E}quivalence of {P}encils of
  {H}ermitian and {A}nti-{H}ermitian {M}atrices}, Amer. J. Math. \textbf{59}
  (1937), no.~2, 399--413.

\end{thebibliography}
\bibliographystyle{amsalpha}

\end{document}